\documentclass[letterpaper, 11pt]{article}
\usepackage{amssymb,amsmath}
\usepackage{epsf}
\usepackage{times,bm,mathrsfs}
\usepackage{amsmath}
\usepackage{amssymb}
\usepackage{amsfonts}
\usepackage{mathrsfs}
\usepackage{bbm}
\usepackage{color}
\usepackage{colordvi}
\usepackage{graphicx}
\usepackage{amscd}
\usepackage{epsfig}
\usepackage{times}

\newtheorem{theorem}{Theorem}
\newtheorem{proposition}{Proposition}

\newtheorem{definition}{Definition}

\newtheorem{lemma}{Lemma}

\newtheorem{claim}{Claim}

\newenvironment{proof}{\noindent{\textbf{Proof:}}}{$\blacksquare$\vskip\belowdisplayskip}

\newenvironment{prevproof}[2]{\noindent {\bf {Proof ({#1}~\ref{#2}):}}}{$\blacksquare$\vskip \belowdisplayskip}

\definecolor{Red}{rgb}{1,0,0}
\definecolor{Blue}{rgb}{0,0,1}
\definecolor{Olive}{rgb}{0.41,0.55,0.13}
\definecolor{Green}{rgb}{0,1,0}
\definecolor{MGreen}{rgb}{0,0.8,0}
\definecolor{DGreen}{rgb}{0,0.55,0}
\definecolor{Yellow}{rgb}{1,1,0}
\definecolor{Cyan}{rgb}{0,1,1}
\definecolor{Magenta}{rgb}{1,0,1}
\definecolor{Orange}{rgb}{1,.5,0}
\definecolor{Violet}{rgb}{.5,0,.5}
\definecolor{Purple}{rgb}{.75,0,.25}
\definecolor{Brown}{rgb}{.75,.5,.25}
\definecolor{Grey}{rgb}{.5,.5,.5}
\definecolor{Black}{rgb}{0,0,0}

\newcommand{\ecal}{\mathcal{E}}

\newcommand{\tcal}{\mathcal{T}}

\renewcommand{\P}{\mathbb{P}}

\renewcommand{\root}{r}
\newcommand{\groot}{\rho}

\newcommand{\ball}[2]{\mathcal{C}_{#1}}

\newcommand{\locations}{\mathcal{X}}
\newcommand{\mrca}{\mathrm{MRCA}}

\renewcommand{\time}{\tau}
\newcommand{\lgt}{\lambda}
\newcommand{\sub}{\mu}

\newcommand{\weight}{\omega}

\newcommand{\maxlgt}{\overline{\lgt}}
\newcommand{\minlgt}{\underline{\lgt}}

\newcommand{\lgtweight}{\Lambda}

\newcommand{\lgttotalextinct}{\boldsymbol{\Lambda}_{\mathrm{tot}}}
\newcommand{\mintime}{\underline{\time}}
\newcommand{\maxtime}{\overline{\time}}

\newcommand{\ratiolgt}{\rho_\lgt}
\newcommand{\ratiotime}{\rho_\time}
\newcommand{\ratiosub}{\rho_\sub}
\newcommand{\minsub}{\underline{\sub}}
\newcommand{\maxsub}{\overline{\sub}}

\newcommand{\forest}{\mathcal{F}}

\newcommand{\remove}[1]{{}}

\addtocounter{page}{-1}

\begin{document}

\title{\vspace{0cm}
Species Trees are Recoverable from Unrooted Gene Tree Topologies Under a Constant Rate of Horizontal Gene Transfer
}

\author{
Constantinos Daskalakis\footnote{Department of Electrical Engineering and Computer Science at MIT. Supported by a Microsoft Research faculty fellowship and NSF Award CCF-0953960 (CAREER) and CCF-110149. This work was done while CD was visiting the Simons Institute for Theoretical Computer Science.}
\and
Sebastien Roch\footnote{Department of Mathematics at the University of Wisconsin--Madison.
Supported by NSF grant DMS-1149312 (CAREER) and DMS-1614242. 
This work
was done while SR was visiting the Simons Institute for Theoretical Computer Science.}
}
\maketitle

\begin{abstract}
Reconstructing the tree of life from molecular sequences is a fundamental problem in computational biology. Modern data sets often contain a large number of genes, which can complicate the reconstruction problem due to the fact that different genes may undergo different evolutionary histories. This is the case in particular in the presence of horizontal genetic transfer (HGT), where a gene is inherited from a distant species rather than an immediate ancestor. Such an event produces a gene tree which is distinct from, but related to, the species phylogeny. 

In previous work, a natural stochastic models of HGT was introduced and studied. It was shown, both in simulation and theoretical studies, that a species phylogeny can be reconstructed from gene trees despite surprisingly high rates of HGT under this model. Rigorous lower and upper bounds on this achievable rate were also obtained, but a large gap remained. Here we close this gap, up to a constant. Specifically we show that a species phylogeny can be reconstructed correctly from gene trees even when, on each gene, each edge of the species tree has a constant probability of being the location of an HGT event. Our new reconstruction algorithm, which relies only on unrooted gene tree topologies, builds the tree recursively from the leaves and runs in polynomial time. 

We also provide a matching bound in the negative direction (up to a constant) and extend our results to some cases where gene trees are not perfectly known.

\end{abstract}

\thispagestyle{empty}

\clearpage

\section{Introduction}\label{section:introduction}

A major challenge in the reconstruction of the
tree of life from modern molecular datasets
is that different genes often tell
conflicting stories 
about the evolutionary history of a group
of organisms~\cite{Maddison:97,DeBrPh:05,Nakhleh:13}.
Consider the example in Figure~\ref{fig:spr}.
\begin{figure}[b!]
	\centering
	\includegraphics[width=5in]{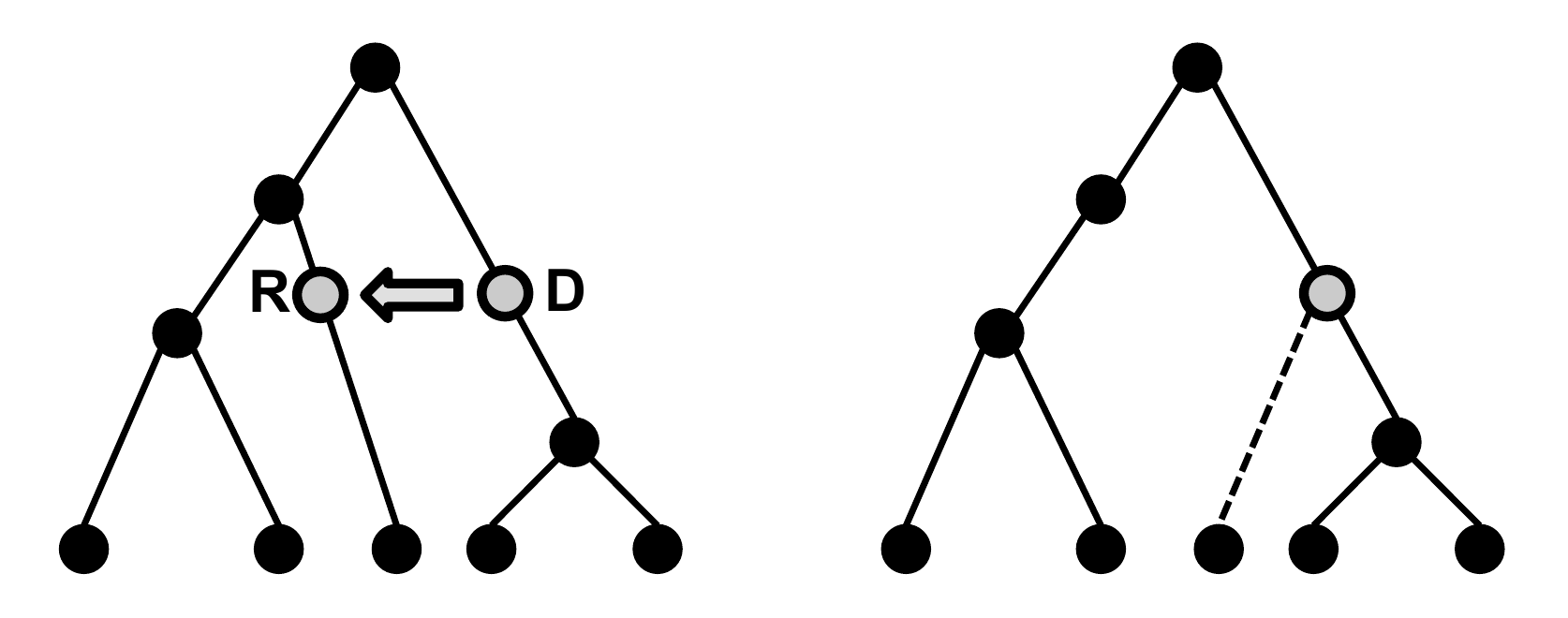}
	\caption{
		An HGT event. On the left, the species
		phylogeny is shown with the donor (D) and
		recipient (R) locations.
		On the right, the resulting (unweighted) gene tree
		is shown after the HGT event.
	}\label{fig:spr}
\end{figure}
On the left-hand side is depicted a \emph{species phylogeny},
where branchings (i.e., internal nodes) represent
past speciation events and leaves correspond to
contemporary species~\cite{SempleSteel:03}.
Imagine that, at some point in the past
history of these species, a gene was transferred 
by a virus from
a donor species (the gray node 
labeled $D$) to a recipient species (the gray node 
labeled $R$), a biological process
known as transduction~\cite{hartl1998genetics},
and eventually replaced that gene in the recipient
species. This event is referred to as 
a \emph{horizontal gene transfer} (HGT). From a phylogenetic point of view, the result of this transfer is that, in the tree representing
the history of this transferred gene, the middle
branch is now more closely related to the right subtree of the root (where the donor species lies) than to
the left subtree of the root---in direct conflict with
the species history. This new tree, depicted on the
right-hand side of Figure~\ref{fig:spr}, is called a \emph{gene tree}.

In datasets comprising multiple genes,
each gene has its own history---embod\-ied
by its own gene tree. An abundance of reconstruction
algorithms have been developed to infer
such gene trees using DNA sequences
extracted from a reference individual of each leaf
specie~\cite{Felsenstein:04}. In particular
much is known about the rigorous, theoretical properties
of these single-gene inference methods. See e.g.,~\cite{SteelSzekely:02,Mossel:03,Mossel:04a,MoRoSl:11,ErStSzWa:99a,DaMoRo:11a,DaMoRo:11b,Roch:10,ErStSzWa:99b,CrGoGo:02,MosselRoch:05,DaskalakisRoch:10,DaskalakisRoch:13,AnDaHaRo:10,AnDaHaRo:12} for a sample of results. 
In some parts of the tree of life, e.g., in bacteria,
it is common for a significant fraction of
gene trees to be in conflict 
with the species phylogeny. See~\cite{ZhGoCh+:06} for an example. Although individual
gene histories are interesting in their own
right, a fundamental goal is
to reconstruct the species history (which is 
not directly measurable), i.e., 
the sequence of speciation events that have
produced the current diversity of species.
(We ignore past extinctions, which cannot be inferred
purely from contemporary data.) Therefore, 
a key problem in modern phylogenetics is:
\emph{how to infer a species phylogeny from a
collection of (possibly discordant) gene trees?}

The answer to this question depends on the
mechanism(s) responsible for the discordances,
which include HGT as well as gene tree estimation errors,
incomplete lineage sorting, gene duplications
and losses, etc.~\cite{Maddison:97}. In
the current work, we focus solely on 
HGT. One possible reconstruction
approach is to identify genes or loci 
which are believed to have undergone
little or no HGT, such as 16S 
ribosomal RNA sequences~\cite{Fox25071980}. 
However such sequences are typically short,
leading to unreliable tree estimates~\cite{SteelSzekely:02,Mossel:03}. Moreover a single-gene approach ignores much of the available data.
Here we consider inference methods based on multiple genes that explicitly model the discordances
produced by HGT. A stochastic model
of HGT was introduced by Roch and Snir~\cite{RochSnir:12}, inspired by work
of Kim and Salisbury~\cite{KimSalisbury:01} and Galtier~\cite{Galtier:07}. In this model,
for each gene independently,
HGT events occur
at random along the phylogeny according to a Poisson point process (see Section~\ref{sec:definitions-and-results} for details). The goal is then
to recover the species phylogeny from 
a collection of gene trees, each of which can be thought
of as a ``randomly scrambled'' instance of the
species phylogeny. A related
model was studied in~\cite{Linz01062007,Steel201381,Sand2013295}.
In particular, Steel et al.~\cite{Steel201381} shed light on some of the challenges arising
in this context by showing that applying a majority rule to all triples of species may fail to recover the correct topology.

A natural question is:
\emph{when is there too much HGT to recover the
species phylogeny?} It was proved
in~\cite{RochSnir:12} that surprisingly
high rates of HGT can in fact be tolerated, in agreement with simulation results of Galtier~\cite{Galtier:07}.
Roughly, under assumptions that will
be detailed in Section~\ref{sec:definitions-and-results}, 
a species phylogeny with $n$ leaves can be recovered
from a logarithmic number of genes when the HGT rate is at most
$O(1/\log n)$ per unit time. 
On the other hand, it was also shown in~\cite{RochSnir:12} that 
there are species phylogenies that \emph{cannot} be distinguished with constant probability
from the same number of genes when the HGT rate is of the order
of $\Omega(\log\log n)$ per unit time.

Here we close the gap. 
That is, under the same assumptions, 
we show that in fact a \emph{constant} rate of HGT
can be tolerated, with a matching bound (up to constants)
in the negative direction.
The algorithmic result in~\cite{RochSnir:12}
is based on the observation that,
under the assumptions made there, 
for any gene the subtree spanned by any
four leaves is unlikely to be the location
of an HGT event when the rate is $O(1/\log n)$. 
By taking a majority vote across genes,
the corresponding subtree of the species phylogeny
can be obtained with high probability. Then,
using standard techniques~\cite{SempleSteel:03},
the full species phylogeny can be derived from
all four-leaf subtrees, also called quartet topologies
or simply quartets. This argument
fails when the rate of HGT is constant. Instead,
we use a recursive approach which progressively
builds the species phylogeny from the leaves up,
using the information obtained from partially
reconstructed subtrees to reach further into the
past. This recursive approach is reminiscent
of recent work in the single-gene context
where tight results were obtained using this type of approach~\cite{Mossel:04a,DaMoRo:06,Roch:10}.
The negative result, on the other hand,
follows from a coupling argument. 

The rest of the paper is organized as follows.
The stochastic model of gene trees under HGT is described
in Section~\ref{sec:definitions-and-results}, 
alongside a statement of the main results.
The proof of the algorithmic result 
is first presented in a special case in Section~\ref{sec:proofs}.
The full proof is then derived in Section~\ref{sec:proofs-contractions}.
The impossibility result is detailed in Section~\ref{sec:proofs-2}.

\section{Definitions and Results}\label{sec:definitions-and-results}

In this section 
	we introduce the stochastic model of 
	horizontal gene transfer (HGT), which
	is based on models of
	Kim and Salisbury~\cite{KimSalisbury:01,GeWaKi:05} and of Galtier~\cite{Galtier:07,GaltierDaubin:08}.
	In essence, we assume that HGT events occur 
	at random along the species phylogeny. 
	We follow roughly the presentation 
	in~\cite{RochS12}. See also~\cite{Suchard:05,Jin-Bioinf-2006,Linz01062007}
	for related models.
	After introducing the model in Section~\ref{sec:model}, we proceed with the formal statement of our main results. 

\paragraph{Notation} Recall that, for functions $f(n), g(n)$,
$f = O(g)$ means that there is constant $C > 0$ such that
$f(n) \leq C g(n)$ for all $n$ large enough. Similarly,
$f = \Omega(g)$ indicates $f(n) \geq C' g(n)$ for $C' > 0$. 
In addition $f = \Theta(g)$ is equivalent to
$f = O(g)$ and $f = \Omega(g)$.
By {\it polynomial in $n$}, we mean $O(n^{C''})$
for some constant $C'' > 0$.
We use the notation $\P[\ecal_0\,|\,\ecal_1]$
for the conditional probability of $\ecal_0$
given $\ecal_1$.

\subsection{Stochastic Model of HGT}
\label{sec:model}

A species phylogeny
is a graphical representation 
of the speciation history of a collection of organisms.
The leaves correspond to extant species. 
Each branching indicates a speciation event. 
To each edge is associated a 
positive value corresponding to the time
elapsed along that edge. 
\begin{definition}[Species phylogeny]\label{def:phylogeny}
A \emph{species phylogeny} 
$T_s = (V_s,E_s;\root,\time)$ 
is a directed tree rooted at $\root$ with vertex set $V_s$, edge set $E_s$ and
$n$ labelled leaves $L = [n] = \{1,\ldots,n\}$ such that 1) the degree of all internal vertices
$V_s-L$ is exactly $3$ except the root $\root$
which has degree $2$, and 2) the edges are assigned 
inter-speciation times 
$\time : E_s \to (0,+\infty)$.
We assume that $T_s$ is \emph{ultrametric},
that is, from every node, 
the path lengths with respect to $\time$ from that node
to all its descendant leaves are equal.
(This is equivalent to assuming that all leaves
	are contemporaneous.)
\end{definition}
\noindent Phylogenies are naturally equipped with a notion
of distance between leaves (or more generally vertices).
Such metrics are useful in reconstructing phylogenies.
\begin{definition}[Species metric]\label{def:species metric}
	A species phylogeny 
	$T_s = (V_s,E_s;\root,\time)$
	induces a metric $\time$ on the leaves defined as follows, for all $u,v \in L$:
	$$\time(u,v) = \sum_{e \in p(u,v)} \time(e),$$
	where $p(u,v)$ is the unique path between leaves $u,v$ in the phylogeny, viewed as a set of edges. We call $\time$ the {\em species metric}.
\end{definition}

To infer a species phylogeny, 
we first reconstruct gene trees,
that is, trees of ancestor-descendant relationships 
for orthologous genes (or, more generally, loci).
Phylogenomic studies have revealed extensive 
discordance between gene trees, in particular,
	as a result of HGT as we describe below (e.g.~\cite{BaSuLe+:05,DoolittleBapteste:07}).
\begin{definition}[Gene tree]\label{def:genetree}
A \emph{gene tree} $T_g = (V_g,E_g;\groot,\weight_g)$
for gene $g$ 
is a directed tree rooted at $\groot$ 
with vertex set $V_g$, edge set $E_g$ and the same leaf-set $L=\{1,\ldots,n\}$ as the species phylogeny
such that 1) the degree of every internal vertex
is {either $2$ or} 
$3$, and 2) the edges are assigned branch lengths $\weight_g : E_g \to (0,+\infty)$. Similarly to Defintion~\ref{def:species metric},
we let $\weight_g(u,v)$ be the sum of the branch lengths
on the path between $u$ and $v$.
\end{definition}

As we will discuss below, gene trees are derived from---
or ``evolve'' on---the species phylogeny. In our model,
their branch lengths represent  expected
numbers of substitutions. Their topology and branch lengths
may differ from those of the species phylogeny as a result of
HGT events.
Our stochastic model
of HGT requires a {\em rooted} species
phylogeny as time plays a key role
in constraining valid HGT events. Indeed such events 
necessarily involve contemporaneous locations in the species phylogeny.
See, e.g.,~\cite{JIN-TCBB-2009}. 
In particular our
results rely on the ultrametricity property of the
 species phylogeny.

We now formalize a stochastic model of HGT. 
First some notation.
Let 
$T_s = (V_s,E_s;\root,\time)$
be a fixed species phylogeny.
By a {\em location} in $T_s$, we mean
any position along $T_s$ seen as a continuous
object, 
that is,
a point $x$ along an edge $e \in E_s$.
We write $x \in e$ in that case.
We denote the set of locations in $T_s$
by $\locations_s$. 
We say that $x \in \locations_s$ is an {\em ancestor}
	of $y \in \locations_s$ if $x$ is on the path between
	$y$ and $r$ in $T_s$ (in which case $y$ is also
	a \emph{descendant} of $x$).
For any two locations $x, y$ in $\locations_s$,
we let $\mrca(x,y)$ be their most 
recent common ancestor (MRCA) in $T_s$
and we let $\time(x,y)$ be the length 
of the path connecting $x$ and $y$
in $T_s$ under the metric naturally defined
by the weights $\{\time(e), e \in E_s\}$,
interpolated linearly to locations
along an edge.
In words $\time(x,y)$, which we refer to
as the $\time$-distance between $x$ and $y$, 
is the sum of times
to $x$ and $y$ from $\mrca(x,y)$. We say that
two locations $x, y$ are 
{\em contemporaneous} if their respective
$\time$-distance to the root $\root$
is identical, that is, 
$$
\time(x,\root) = \time(y,\root).
$$ 
We let 
$$
\ball{x}{R}
= \{y \in \locations_s : \time(\root,x) = \time(\root,y)\}$$ 
be the
set of locations contemporaneous to $x$.

We associate to each edge $e\in E_s$ in $T_s$ 
a {\em rate of horizontal gene transfer} $0 < \lgt(e) < +\infty$.
We let 
$\lgtweight(e) =  \lgt(e) \time(e)$, $e \in E_s$.
We note that, since $\lgt(e)$ is the HGT rate on $e$,
$\lgtweight(e)$ gives the expected number of 
HGT events along $e$. Further, we let
$$
\lgttotalextinct = \sum_{e\in E_s} \lgtweight(e),
$$
be the {\em total HGT weight} of the phylogeny. 

Our model of HGT is as follows.
From a topological point of view, 
an HGT event is
equivalent to a subtree-prune-and-regraft (SPR)
operation~\cite{SempleSteel:03}.
The recipient location, that is, the location receiving the
gene transfer, is the point of pruning.
The donor location is the point of
regrafting. 
In other words, 
on the gene tree, a new internal node is created 
at the donor location with two children nodes, 
one being the 
original endpoint of the corresponding edge 
and the other being the node immediately 
under the recipient location in the species phylogeny. 
The original edge going to the latter node is removed.
Refer to Figure~\ref{fig:spr} for an illustration.

Before describing the model formally,
	we need some further notation.
	As will become clear from the description
	of the HGT process below (see the example in Figure~\ref{fig:spr}), each edge $e$ of
	the gene tree $T_g$ corresponds to a full or a partial 
	edge of the species phylogeny $T_s$. 
	In particular, there exists a mapping 
	$(\eta,\zeta_b,\zeta_f): E_g \rightarrow E_s \times \mathbb{R}_+ \times \mathbb{R}_+$,
	mapping an edge $e \in E_g$ to an edge $\eta(e) \in E_s$ 
	and a pair of times $0 \le \zeta_b(e) \le \zeta_f(e) \le \time(\eta(e))$.
	The quantities
	$\zeta_b(e)$ and $\zeta_f(e)$  represent times of HGT events on edge $\eta(e)$, as we will
	define below. 
	Finally, for each gene $g$ and each edge
		$e \in E_s$ in the species tree, we associate a 
		rate of substitution $0 < \sub_g(e) < +\infty$.
\begin{definition}[Stochastic model of HGT]
\label{def:randomlgt}
Let 
$T_s = (V_s,E_s;\root,\time)$
be a fixed spec\-ies phylogeny.
A gene tree $T_g$ is generated according
to the following continuous-time stochastic process,
which gradually modifies the species phylogeny 
starting at the root. There
are two components to the process:
\begin{enumerate}
\item {\bf HGT locations.} 
The recipient and donor locations
of HGT events are selected as follows:
\begin{itemize}
\item {\em Recipient locations.}
Starting from the root,
along each branch $e$ of $T_s$, 
locations are selected as recipient of a gene tranfer
according to a continuous-time Poisson
process with rate $\lgt(e)$. 
Equivalently, the total number of 
HGT events is Poisson with mean $\lgttotalextinct$
and each such event is located independently
according to the following density.
For a location $x$ on branch $e$, 
the density at $x$ is
$\lgtweight(e)/\lgttotalextinct$.

\item {\em Donor locations.} If $x$ is selected
as a recipient location, the corresponding
donor location $y$ is chosen uniformly at random
in $\ball{x}{R}$. The HGT transfer is then
obtained by performing an SPR move from $x$ to $y$, 
that is, the subtree below $x$ in $T_s$ is moved
to $y$ in $T_g$. 
\end{itemize}
The probability that a recipient or donor location coincides with a node of $T_s$ is $0$. 

\item {\bf Executing the HGT Process:} We perform gene transfers  
chronologically from the root:
\begin{itemize}
\item We initialize the gene tree as follows: $V_g=V_s$, $E_g = E_s$. 
\item We also initialize the mappings $(\eta,\zeta_b,\zeta_f)$ as follows, for all $e \in E_g$: $\eta(e)=e$; $\zeta_b(e)=0$; $\zeta_f(e)=\time(e)$.
\item We process the HGT events chronologically as follows:
\begin{enumerate}
\item Suppose the next event to process has $x \in e \in E_s$ as recipient location and $y \in e' \in E_s$ as donor location.
\item We find the unique edges $e_x, e_y \in E_g$ such that:
\begin{itemize}
\item $\eta(e_x)=e$ and $\eta(e_y)=e'$; and
\item $\zeta_b(e_x) \le \time_x \le \zeta_f(e_x)$ and $\zeta_b(e_y) \le \time_y \le \zeta_f(e_y)$;
\end{itemize}
where $\time_x$ is the time between $x$ and its most recent ancestor in $T_s$, and similarly for $\time_y$.
\item We introduce a new node $v$, splitting $e_y$ into two consecutive edges, $e_{y_1}$ and $e_{y_2}$, with the following features:
\begin{itemize}
\item $\eta(e_{y_1})=\eta(e_{y_2})=e'$;
\item $\zeta_b(e_{y_1})= \zeta_b(e_{y})$; $\zeta_f(e_{y_1})= \time_y$;
\item $\zeta_b(e_{y_2})= \time_y$; $\zeta_f(e_{y_2})= \zeta_f(e_{y})$.
\end{itemize}
\item If $e_x=(u,w)$, we update it to $e_x=(v,w)$, and change $\zeta_b(e_x)=\time_x$.
\end{enumerate}
\end{itemize}
 
\end{enumerate}
After all HGT events have been processed, the weights on the resulting gene tree $T_g$ are defined as follows. 
For all $e \in E_g$, $\weight_g(e) = (\zeta_f(e)-\zeta_b(e)) \cdot \sub_g(\eta(e))$.
\end{definition}

Observe that HGT events may disconnect
		subtrees of the species phylogeny from their original roots, connecting them
		to other branches of the gene tree, thereby
		creating nodes of degree $2$ in the gene tree. We allow internal vertices
		of degree $2$ in a gene tree to potentially
		delineate between two consecutive species phylogeny edges.
	Each gene tree branch length $\weight_g(e)$ represents
	the expected number of substitutions
	on the (possibly partial) edge of the species phylogeny corresponding
	to edge $e$ of the gene tree, which
	is determined by the substitution rate $\sub_g(\eta(e))$, as well as the times $\zeta_b(e)$ and
	$\zeta_f(e)$.

\subsection{Species phylogeny reconstruction under constant HGT rate} 
\label{sec:results}

Let 
$T_s = (V_s,E_s;\root,\time)$ 
be an unknown species phylogeny. 
We assume that $N$ independent gene trees $T_{g_1},\ldots,T_{g_N}$,
corresponding to homologous genes $g_1,\ldots,g_N$, were generated according to
the process of Definition~\ref{def:randomlgt}. \emph{Our overall goal is to reconstruct the species phylogeny, given the gene trees.} 

\paragraph{Problem statement}
However, given that the gene trees are ultimately reconstructed from genetic sequences, we assume that we have \emph{imperfect}
knowledge of these trees. 
To formalize this further, we make the following definitions.
\begin{definition}[Real subtree] \label{def:compound definition}
Given a rooted tree $T$, we call a subtree $T'$ of $T$ {\em real} if all leaves of $T'$ are leaves of $T$. 
Given a node $u$ of $T$, we denote by $u \downarrow T$ the subtree of $T$ rooted at $u$. 
\end{definition}
\begin{definition}[Leafsomorphic trees]
Given two leaf-labeled rooted, directed trees $T=(V,E)$ and $T'=(V',E')$ we call them {\em leafsomorphic}
if there exists a leaf-label respecting isomorphism between the trees $\tilde{T}$ and $\tilde{T'}$ obtained from $T$ and $T'$ respectively, after replacing all maximal directed paths $\langle u, u_1,\ldots,u_k, v \rangle $ whose internal vertices have in- and out-degree $1$ by a single directed edge $\langle u,v\rangle$.	
\end{definition}
With the above definitions we can formalize the information that our algorithm is given. 
\begin{itemize}
\item {\bf Contracted unrooted gene tree topologies.}
We assume that we are given
unrooted gene tree topologies where only those edges on a path between two leaves are kept and degree $2$ vertices are suppressed. In addition,
we note that the HGT process can 
produce gene tree branch lengths
that are arbitrarily short, and therefore, that may be
hard to reconstruct from DNA sequences.
Hence
we also assume that a subset of edges (possibly all) whose length
is below a threshold $\epsilon$ are contracted.
Namely, for each gene $g$, we are given an $\epsilon$-contraction of its gene tree, defined as follows.
\begin{definition}[$\epsilon$-Contraction] \label{def:contraction of gene tree}
	For $\epsilon\geq 0$, an $\epsilon$-{\em contraction of a gene tree} $T_g = (V_g,E_g;\root,\weight_g)$ is a unrooted, unweighted tree topology $T_{g}'=(V_g', E_g')$ on the same set of leaves $L$ as $T_g$ obtained from the following construction. We start by unrooting $T_g$, removing all edges of $T_g$ that are not on a path between two leaves, and replacing all maximal paths $(u, u_1,\ldots,u_k, v)$ whose internal vertices have 
	degree $2$ by a single edge $(u,v)$ whose weight is the sum of the weights
	of the edges on the path. We then contract
	a subset of edges (possibly all) 
	whose weigth is $\leq \epsilon$,
	i.e., for each chosen edge, we remove the edge and fuse its endpoints into a single new vertex. We then discard
	the weights. The result is $T_g'$.
	\end{definition}
\noindent See~\cite{DaMoRo:11b} for a reconstruction
algorithm that produces such a contraction from
sequences at the leaves. We use the notation $d_g(u,v)$ and $d_g'(u,v)$ for the graph distances between $u$ and $v$
on $T_g$ and $T_g'$ respectively.
\end{itemize}
\noindent The {\bf Species Phylogeny Reconstruction Problem in the Presence of HGT} is the following:
\begin{quote}
Given $\epsilon$-contractions $T_{g_1}',\ldots,T_{g_N}'$ of $N$ independent gene trees generated under the process of Definition~\ref{def:randomlgt}, reconstruct the topology $\tcal_s$ of the phylogeny, namely reconstruct the rooted tree $(V_s,E_s)$ up to a leaf-label respecting isomorphism.
\end{quote}
\noindent Our main focus in this work is on the rate of HGT that can be sustained without obscuring the phylogenetic signal.

\paragraph{Main result}
To derive asymptotic results, we make some assumptions on the underlying model.
The following
assumptions were introduced in~\cite{DaskalakisRoch:10,DaskalakisRoch:13}
and are related to commonly made assumptions in the
mathematical phylogenetics literature.
\begin{definition}[Bounded-rates model]
\label{def:brm}
Let $0 \leq \ratiolgt \leq 1$, $0 < \ratiotime, \ratiosub \leq 1$,
 and $0 <  \maxtime,\maxlgt,\maxsub < +\infty$ be constants.
Under the bounded-rates model, we consider
the set of phylogenies 
$T_s = (V_s,E_s;\root,\time)$ 
on $n$ extant leaves with rates of transfer $\lgt(e)$ and
rates of substitution $\sub_g(e)$
such that the following conditions are satisfied: $\forall e\in E_s$ and all genes $g$, 
$\minlgt \equiv \ratiolgt \maxlgt \leq \lgt(e) \leq \maxlgt$, $\mintime \equiv \ratiotime \maxtime \leq \time(e) \leq \maxtime$, and
$\minsub \equiv \ratiosub \maxsub \leq \sub_g(e) \leq \maxsub$.
\end{definition}
Finally, our main result is the following.
\begin{theorem}[Algorithmic result: $\epsilon$-contractions]
	\label{thm:main1b}
	Fix constants $0 \leq \ratiolgt \leq 1$, $0 < \ratiotime \leq 1$, $0 < \ratiosub \leq 1$, $0 <  \maxtime,\maxsub < +\infty$ and $0 \leq \epsilon < \mintime\minsub$.
	Under the bounded-rates model,
	it is possible to reconstruct 
	the topology of the species phylogeny 
	with probability at least $1-{1 \over {\rm poly}(n)}$
	from 
	$\epsilon$-contractions of $N= \Omega(\log n)$ independent gene trees generated under the process
	of Definition~\ref{def:randomlgt}, as long as
	$\maxlgt$ is a sufficiently small constant not depending on $n$.
\end{theorem}
\noindent Our reconstruction algorithm is detailed in the proof of Theorem~\ref{thm:main1b}. The condition on $\epsilon$
in Theorem~\ref{thm:main1b} corresponds to
the requirement that branches with no transfer
are present in the gene tree, that is, are 
not contracted.


		
\paragraph{Proof sketch}
We first prove the result in an easier case,
the ultrametric case with partial branch length information.\footnote{Theorem~\ref{thm:main1} was announced without proof in extended abstract form in~\cite{DaskalakisRoch:16}.}
That is, we assume that the rate of substitution satisfies $\sub_g(e) = \sub$ for all $e$ for some $\sub > 0$ and that we are given an
$\epsilon$-distortion of the resulting gene tree metric.
\begin{definition}[$\epsilon$-Distortion] \label{def:distortion of gene tree}
	For $\epsilon \geq 0$, an $\epsilon$-{\em distortion of a gene tree} $T_g = (V_g,E_g;\root,\weight_g)$ is a rooted, directed tree $T_{g}'=(V_g', E_g'; \root',\weight_g')$ on the same set of leaves $L$ whose internal vertices have out-degree $2$
	which is obtained as follows. We remove all edges of $T_g$ that are not on a path between two leaves and replace all maximal paths $(u, u_1,\ldots,u_k, v)$ whose internal (non-root) vertices have 
	degree $2$ by a single edge $(u,v)$ whose weight is the sum of the weights
	of the edges on the path.
	Moreover, the edge weights $\weight_g'$ of $T_g'$ define a metric on the leaves that is $\epsilon$-close to the metric defined by $\weight_g$, namely for all pairs of leaves $v,w \in L$:
	$|\weight_g(v,w)-\weight_g'(v,w)| \le \epsilon$.
\end{definition}
\begin{theorem}[Algorithmic result: ultrametric $\epsilon$-distortions]
	\label{thm:main1}
	Fix constants $0 \leq \ratiolgt \leq 1$, $0 < \ratiotime \leq 1$, $0 <  \maxtime,\sub < +\infty$
	and $0 \leq \epsilon < \frac{\mintime\sub}{2}$. 
	Under the bounded-rates model where we further assume that $\sub_g(e) = \sub$ for all $e$,
	it is possible to reconstruct 
	the topology of the species phylogeny 
	with probability at least $1-{1 \over {\rm poly}(n)}$
	from 
	$\epsilon$-distortions of $N= \Omega(\log n)$ independent gene trees generated under the process
	of Definition~\ref{def:randomlgt}, as long as
	$\maxlgt$ is a sufficiently small constant not depending on $n$.
\end{theorem}
\noindent 
Our reconstruction algorithm, which is detailed in the proof of Theorem~\ref{thm:main1}, is recursive: it reconstructs the species phylogeny a few ``levels'' from the leaves at a time. 
To give some insights into how it works, 
we first observe that it is infeasible to use the approach of~\cite{RochS12} under the conditions
of Theorem~\ref{thm:main1}.	Indeed, in~\cite{RochS12}, 
the induced species phylogeny topology on every subset of four leaves $\{a,b,c,d\} \subset L$, also known as a {\em quartet}, is determined directly by using the majority induced topology on these four leaves across gene trees. When HGT rates are low enough, it can be shown that most such induced gene tree topologies 
coincide with the species phylogeny~\cite{RochS12}.
The full species phylogeny can then be reconstructed
from the collection of all quartets using
standard techniques (see e.g.~\cite{SempleSteel:03}).
But Theorem~\ref{thm:main1} allows an expected $\Omega(\log n)$ HGT events on every path from the root to a leaf, making the argument
in~\cite{RochS12} invalid. 
Instead, we work our way up the tree, obtaining stronger evidence for the state of quartets as we get firmer knowledge of the lower levels of the tree.
A related approach has proved very powerful
in the context of phylogeny reconstruction from
a single gene (see e.g.~\cite{Mossel:04a,DaMoRo:11a}).

The proof of Theorem~\ref{thm:main1} contains several steps:
\begin{enumerate}
	\item {\bf Reconstructing the recent past:}
	We first show how to use pairwise
	distance information to reconstruct the species
	phylogeny in the ``recent past.'' The basic
	idea is to show that, for each pair of 
	leaves at ``short distance'' in the species
	phylogeny, the median distance across all
	genes is a good estimate of the actual distance
	in the species phylogeny (Lemma~\ref{lem:one step}). We then use
	standard distance-based techniques to reconstruct
	the shallow part of the species phylogeny (Lemma~\ref{lem:trees-from-distorted-metrics}).
	
	\item {\bf Going deeper into the tree:} We then
	bootstrap the previous argument to reach
	deeper parts of the species phylogeny. The main
	problem is to identify corresponding vertices
	in the gene trees and in the reconstructed parts
	of the species phylogeny. Because of the extensive
	HGT, such a task is far from trivial. We show that,
	for each vertex at the frontier of the reconstructed
	phylogeny and for each gene, 
	one can find with high probability a certain type of subtree rooted at the corresponding vertices,
	called a diluted subtree, which has not undergone
	HGT and, therefore, is shared by the gene tree
	and the species phylogeny (Lemma~\ref{lem:exists-diluted-subtree}). We then
	show how to use such diluted subtrees to estimate
	the distance between close-by pairs of vertices 
	deep inside the reconstructed phylogeny (Proposition~~\ref{lem:generalized one step}).
	
	\item {\bf Computing diluted trees and recursing:} We show how
	to compute diluted subtrees in Proposition~\ref{lem:identifying diluted tree}.
	The algorithm is based on a dynamic programming
	approach. 
	The final details of the proof
	are described in Section~\ref{sec:proof-finish}
	where the main induction step is implemented.
	
\end{enumerate}

To prove Theorem~\ref{thm:main1b},
we make use of graph distances rather than distortions
and 
we employ 
a related ``unrooted'' approach, cherry picking, which is detailed
in Section~\ref{sec:proofs-contractions}.

\subsection{An impossibility result}

We also provide evidence that the reconstruction problem becomes significantly harder when the HGT rate is larger than a high enough constant. Specifically, we show that conisderably more data is needed in that regime.
\begin{theorem}[Impossibility result]
\label{thm:main2}
Fix $\rho_\lgt = 0$.
Under the bounded-rates model, for all $\ratiotime$,
$\ratiosub$, $\maxtime$ and $\maxsub$, 
there is a constant $\bar{\lambda}$ large enough such that
for any $n$ there exists two species phylogenies
which produce the same $N = \Omega(n^{1/6})$ gene trees with 
probability at least $1/2$.
\end{theorem}
\noindent The proof uses a coupling argument which is presented in Section~\ref{sec:proofs-2}. We point out that we were unable to obtain a provably correct reconstruction algorithm in this regime, {\it even assuming that the number of genes satisfies the conditions of Theorem~\ref{thm:main2}---or, in fact, even if $N = +\infty$.}
In particular, the question of the identifiability of the model remains an outstanding open problem in this area.

\section{Algorithmic result: $\epsilon$-distortions}
\label{sec:proofs}

In this section we provide the proof of Theorem~\ref{thm:main1}. 
In particular, we assume that $\sub_g(e) = \sub$ for all $e$.
Throughout this section, our {\bf Operating Assumptions} are the following:
We are given $\epsilon$-distortions $T_{g_1}',\ldots,T_{g_N}'$ of gene trees $T_{g_1},\ldots,T_{g_N}$, generated independently according to the random HGT model of Definition~\ref{def:randomlgt} from a species phylogeny $T_s = (V_s,E_s;\root,\time)$ with rates of horizontal transfer $\lgt(e)$ 
and a constant rate of substitution $\sub$ satisfying the bounded rates model of Definition~\ref{def:brm}. We assume in particular that $\epsilon < {\sub \cdot \mintime \over 2}$. Additionally $N \ge C \log n$ for a large enough constant $C$, and $\maxlgt$ a small enough constant, as required by all the lemmas established in this section. In particular, we will skip stating these assumptions in the statements of all lemmas.
To simplify the notation, we let
$$
\weight_s(u,v) = \sub\cdot\time(u,v).
$$
The proofs of the lemmas below can be found
in Section~\ref{sec:proofs-lemmas-distorted}.

\subsection{Reconstructing the Recent Past}

In this section, we show that the signal from the distorted gene trees is strong enough to reconstruct the recent past from the leaves of the species phylogeny. 
We use a distance-based approach.
The key observation, encapsulated in the following
lemma, is that median distances provide accurate estimates
of ``short distances.'' This intuitively
follows from the fact that, at small enough
rates of HGT, the path between two close-by
leaves is unlikely to be the site of an HGT event. 
In fact, the lemma
says a bit more: median distance estimates of 
\emph{long} distances are also guaranteed to exceed a 
threshold.
\begin{lemma}[Median distances are accurate estimates
	of short distances] \label{lem:one step}
For any constant $d_0 >0$, under our operating assumptions, for all $u,v \in L$, the following are true with probability at least $1-{1\over {\rm poly}(n)}$:
\begin{enumerate}
\item \label{lem:one step:close} \emph{Short distances.} 
If $\weight_s(u,v)\le d_0$, then 
${\rm median}_{i =1,\ldots,N}\{\weight_{g_i}'(u,v) \} = \weight_s(u,v) \pm  \epsilon;$
\item \label{lem:one step:far} \emph{Long distances.} If $\weight_s(u,v)>d_0$, then 
${\rm median}_{i =1,\ldots,N}\{\weight_{g_i}'(u,v) \} > d_0 -  \epsilon.$ 
\end{enumerate}
\end{lemma}

How do we use Lemma~\ref{lem:one step} to reconstruct
the recent past?
Let us first formalize what we mean by the ``recent past.''
In essence, we truncate the species phylogeny at a 
fixed time in the past---thus producing a forest. However, because of the
distorted nature of our input, such a truncation must
be defined with care.
We will need the following notation. Given a rooted tree $T$ and a subset of its leaves $L'$, we denote by $T|L'$ the restriction of $T$ to leafset $L'$, i.e., the smallest connected subgraph of $T$ that contains $L' \cup \{\mrca(L')\}$.
\begin{definition}[Truncation of a phylogeny]
	\label{def:truncation}
	Given a phylogeny $T_s = (V_s,E_s;\root,\time)$ with leaf-set $L=[n]$ and some $D>\epsilon>0$, a {\em $(D,\epsilon)$-truncation of $T_s$} is a leaf-labeled forest $T_s^D=(V_s',E_s')$ with leaf-set $L=[n]$, satisfying the following properties:
	\begin{itemize}
		\item \emph{Disjoint forest.} For some $k \le n$, $T_s^D$ comprises $k$ rooted trees, with disjoint leaf-sets $L_1,\ldots,L_k$ which, further, correspond to clusters in the species phylogeny, that is, for all $1\le i \le j \le k$: $\mrca_{T_s}(L_i \cup L_j) \neq \mrca_{T_s}(L_i), \mrca_{T_s}(L_j).$ 
		
		\item \emph{Truncation.} Every pair of leaves $u, v \in L$, such that $\weight_s(u,v) \le D-2\epsilon$ belong to the same $L_i$, and every pair of leaves $u,v \in L$ such that $\weight_s(u,v)>D$ belong to different $L_i$'s.		
		
		\item \emph{Faithfulness.} For all $i=1,\ldots,k$, the leaf-labeled tree $T_s^D|L_i$ is isomorphic to the leaf-labeled tree $T_s|L_i$, under a leaf-label respecting isomorphism.

	\end{itemize}
\end{definition}
To reconstruct a truncation of the species
	phylogeny, we appeal to standard
dist\-ance-based concepts. See in particular~\cite{KiZhZh:03,Mossel:07}. 
We first recall a well-known approach for reconstructing ultrametric species trees. 
	An ultrametric tree naturally defines a system
	of nested clusters, sometimes called clades (see, e.g.,~\cite{SempleSteel:03}). Indeed, for each vertex
	$v$ in a species phylogeny $T_s$, consider the set $A_v$ of all leaves below $v$, that is, leaves for which $v$ is an ancestor. For all pairs of vertices $u, v$ in $T_s$, we have that either $A_u \cap A_v = \emptyset$ (neither $u$ nor $v$ is an ancestor of the other one),
	$A_u \subseteq A_v$ ($v$ is an ancestor of $u$) or $A_v \subseteq A_u$ ($u$ is an ancestor of $u$). 
	We say that such sets are \emph{nested}.
	Reconstructing the topology of $T_s$ is equivalent
	to reconstructing this system of nested clusters.
	If one is given a species metric $\weight_s$, obtaining
	these clusters is straightforward using, for instance,
	single-linkage clustering: iteratively join the closest pair of reconstructed clusters, where the distance between two clusters is defined as the shortest distance between their respective elements. 
	
	However, we are not given $\weight_s$---what we have is an estimate that is reliable only over short distances 
	\begin{equation}\label{eq:dhat}
	\forall u,v \in L,\qquad \hat{d}(u,v) := {\rm median}_{i =1,\ldots,N}\{\weight_{g_i}'(u,v) \}.
	\end{equation}
	Moreover, we only seek to reconstruct a truncation of the species phylogeny. We explain how to do this in the next lemma.
\begin{lemma}[Building a truncation] \label{lem:trees-from-distorted-metrics}
Assume that $\hat{d}$, as defined in~\eqref{eq:dhat}, satisfies the statement of
Lem\-ma~\ref{lem:one step} with 
$\epsilon < {\mintime \cdot \sub \over 2}$ for some $d_0$. Then a $(d_0,\epsilon)$-truncation of $T_s$ can be computed in polynomial-time (e.g., by single-linkage clustering). 
\end{lemma}


%
%

\subsection{Reaching Deeper into the Past} \label{sec:deeper}

 	Our goal in this section is to reach deeper into
 	the species phylogeny. The minimum distance
 	scheme in Formula~\eqref{eq:dhat-clusters} 
 	is unfortunately not accurate beyond a large constant. 
 	Instead, our basic idea is to bootstrap the median estimator in~\eqref{eq:dhat}. \emph{However, there is
 	a significant hurdle.} Although the leaves
 	of a gene tree and of the species phylogeny
 	trivially match, the same does not hold deeper
 	into the past because of the extensive HGT observed
 	under the rates we consider here. Rather 
 	we introduce a notion of ``conserved'' subtrees.
 	For this purpose, we borrow a combinatorial concept of diluted subtrees from~\cite{Mossel:01}.
 	We use diluted subtrees to show that,
 	for any given gene and any given internal
 	vertex of the species phylogeny, with probability close to $1$
 	there is a ``dense'' subtree of the species
 	phylogeny which has not been modified by the HGT
 	process and, therefore, is shared between
 	the gene tree and the species phylogeny.
\begin{definition}[Diluted subtree]\label{def:diluted tree}
Let $T$ be a binary tree rooted at $\root$. A subtree $T'$ of $T$ is called a {\em diluted subtree of $T$} if $T'$ is rooted at $\root$ and, for all nodes $u$ in both $T'$ and $T$, if $u$ is at (topological) depth 
the $\ell$ from $\root$ with $\ell~{\rm mod}~3=0$, then the number of descendants of $u$ at depth $\ell+1$ in $T$ and $T'$ are equal, the number of descendants of $u$ at depth $\ell+2$ in $T$ and $T'$ are also equal, and the number of descendants of $u$ at depth $\ell+3$ in $T$ and $T'$ are within $1$.
\end{definition}  
\begin{definition}[Containing a diluted subtree]
Given a leaf-labeled tree $T$ rooted at $u$, we say that a leaf-labeled rooted tree $T'$ {\em contains a diluted subtree of $T$} if a real subtree of $T'$ is leafsomorphic to a diluted subtree of $T$.
\end{definition}
\begin{lemma}[Conserved subtrees]\label{lem:exists-diluted-subtree}
Consider the leaf labeled tree $u \downarrow T_s$, rooted at some node $u \in V_s$ of phylogeny $T_s$, and a gene tree $T_g$ generated from $T_s$ according to the process of Definition~\ref{def:randomlgt}. For all $\delta>0$, under our operating assumptions, $T_g$ contains a diluted subtree of $u \downarrow T_s$, with probability at least $1-\delta$. In particular, with probability at least $1-\delta$ for any given gene, there exists a diluted subtree of $u \downarrow T_s$ that does not receive any recipient locations during the HGT process of Definition~\ref{def:randomlgt}.
\end{lemma}

With the concept of a diluted subtree and Lemma~\ref{lem:exists-diluted-subtree}, we can generalize Lemma~\ref{lem:one step} to the following statement. Note that this proposition is only
	existential. We show how to actually \emph{compute} the diluted subtrees and the corresponding cluster distances in the next subsection. We extend
	$\weight_s$ to clusters as before. That is, letting
	$u$ and $v$ be vertices in $T_s$ neither of which
	is an ancestor of the other and letting
	$A_u$ and $A_v$ be the corresponding clusters (i.e., descendant leaves),
	we have
	$
	\weight_s(A_u, A_v) = \min_{a \in A_u, b \in A_v} \weight_s(a,b).
	$
\begin{proposition}[Induction step: Diluted subtrees and distance estimates] \label{lem:generalized one step} 
Con\-sider
constants $d_0, \eta>0$ and a pair of nodes $u, v$ of the phylogeny $T_s$, neither of which is an ancestor of the other. Under our operating assumptions, a distorted gene tree $T_g'$ satisfies the following with probability at least $1-\eta$:
\begin{itemize}
\item \emph{Diluted subtree at $u$.} $T_g'$ contains a real subtree $T_u'$ rooted at some node $u'$ that is leafsomorphic to a diluted subtree of $u \downarrow T_s$; moreover, any such subtree $T_u''$ has the same root;
\item \emph{Diluted subtree at $v$.} $T_g'$ contains a real subtree $T_v'$ rooted at some node $v'$ that is leafsomorphic to a diluted subtree of $v \downarrow T_s$; moreover, any such subtree $T_v''$ has the same root.
\end{itemize}
Moreover, for any such subtrees $T_u''$ and $T_v''$:
\begin{itemize}
\item 
If $\weight_s(u,v)\le d_0$, $\weight_g'(\ell_1,\ell_1') = \weight_s(A_u,A_v) \pm  \epsilon$
	for any leaves $\ell_1$ of $T_u''$ and $\ell_1'$ of $T_v''$; \label{lem:general step:close}
\item 
If $\weight_s(u,v)>d_0$, $\weight_g'(\ell_1,\ell_1') > \weight_s(A_u,A_v) - \weight_s(u,v)+ d_0 -  \epsilon$ for any 
leaves $\ell_1$ of $T_u''$ and $\ell_1'$ of $T_v''$. \label{lem:general step:far}
\end{itemize}
\end{proposition} 

\subsection{Computing Diluted Subtrees}

It remains to show how to compute diluted subtrees.
\begin{proposition}[Induction step: Computing diluted subtrees] \label{lem:identifying diluted tree}
Given a leaf-lab\-eled tree $T$ rooted at $u$ and another leaf-labeled tree $T'$ rooted at $u'$, where both $T$ and $T'$ have the same leaf-set $L$ and they both have internal nodes of outdegree $2$, we can identify in polynomial-time a real subtree of $T'$ that is leafsomorphic to a diluted subtree of $T$, if such a subtree exists in $T'$.
\end{proposition}

\subsection{Theorem~\ref{thm:main1}}
\label{sec:proof-finish}

Using Propositions~\ref{lem:generalized one step} and~\ref{lem:identifying diluted tree}, we are now ready
to prove Theorem~\ref{thm:main1}.

\smallskip

\begin{prevproof}{Theorem}{thm:main1}
For every pair $u, v \in V_s$ neither of which is an ancestor of the other, it follows from Proposi\-tion~\ref{lem:generalized one step} and standard concentration inequalities~\cite{MotwaniRaghavan:95} that, with probability at least $1-{1 \over {\rm poly}(n)}$:
\begin{equation}\label{eq:final2}
\weight_s(u,v)\le d_0 \implies
	{\rm median}_{i \in N_{u,v}}\{\weight_{g_i}'(\ell_i,\ell_i') \} = \weight_s(A_u,A_v) \pm  \epsilon;
\end{equation} 
\begin{equation}\label{eq:final1}
\weight_s(u,v)>d_0 \implies
{\rm median}_{i \in N_{u,v}}\{\weight_{g_i}'(\ell_i,\ell_i') \} > \weight_s(A_u,A_v) - \weight_s(u,v) + d_0 -  \epsilon;
\end{equation} 	
where $N_{u,v}$ is the subset of distorted gene trees that contain a diluted subtree of $u \downarrow T_s$ and of $v \downarrow T_s$. For every such gene tree we let $\ell_i,\ell_i'$ be arbitrary leaves of subtrees that are leafsomorphic to a diluted subtree of $u \downarrow T_s$ and $v \downarrow T_s$ respectively.
Since there are $O(n^2)$ pairs of $u,v \in V_s$, by a union bound, Equations~\eqref{eq:final2} and~\eqref{eq:final1} simultanenously hold for all pairs of $u,v \in V_s$, with probability at least $1-{1 \over {\rm poly}(n)}$. We condition on this event.

We now describe our high-level reconstruction
algorithm. We proceed similarly to the proof of
Lemma~\ref{lem:trees-from-distorted-metrics},
although we employ a slightly different implementation
of single-linkage clustering.
But, instead of the update formula~\eqref{eq:dhat-clusters},
whenever a new cluster is formed, we compute
a diluted subtree of the corresponding tree
and use it to estimate inter-cluster distances
using the median as above. More precisely:
\begin{enumerate}
	\item 
	Let $\mathcal{F} = \{\{u\}\,:\, u \in [n]\}$, set $\hat{d}$ as in~\eqref{eq:dhat} 
	and, for all $u \in [n]$, let $T_{\{u\}}$ be the tree
	composed of only $u$. 
	
	\item Until $\mathcal{F} = \{[n]\}$:
	\begin{enumerate}
		\item Let $A, B$ be two clusters in $\mathcal{F}$ achieving the minimum $\hat{d}$-distance.
		\item Update $\mathcal{F}$ by removing
		$A,B$ and adding $A\cup B$. 
		\item Let $T_{A\cup B}$ be the tree corresponding to the cluster $A \cup B$.
		Let $\rho_{A\cup B}$ be the
		root of $T_{A\cup B}$.
		\item For each gene $i \in N$, compute 
		a real subtree $\tilde{T}_{A\cup B}^i$ of $T_{g_i}$ that is leafsomorphic to a diluted subtree of $T_{A\cup B}$, if such a subtree exists,
		as detailed in the proof of Proposition~\ref{lem:identifying diluted tree}.
		Let $\ell_{A\cup B}^i$ be an arbitrary leaf of $\tilde{T}_{A\cup B}^i$.
		\item Update: for each
		$F \in \mathcal{F}$ with $F \neq A\cup B$,
		set
		$
		\hat{d}(F,A\cup B)
		:= {\rm median}_{i \in N_{\rho_F,\rho_{A\cup B}}}\{\weight_{g_i}'(\ell_{F}^i,\ell_{A\cup B}^i)\}. 
		$
	\end{enumerate}
\end{enumerate}

Arguing as in Lemma~\ref{lem:trees-from-distorted-metrics}
and using~\eqref{eq:final2} and~\eqref{eq:final1},
it follows that running the above algorithm
up to any distance $D$ produces a $(D,\epsilon)$-truncation
of $T_s$. That concludes the proof.
\end{prevproof}

\subsection{Proofs}
\label{sec:proofs-lemmas-distorted}

\begin{prevproof}{Lemma}{lem:one step}
	First, we will need the following claim. 
	For a pair of locations $x,y\in\locations_s$,
	we let $p_s(x,y)$ be the path between $x$ and $y$.
	\begin{claim}\label{claim:no-transfer}
	For all $\time^* > 0$ and $\delta^* < 1$, there is a
	$\maxlgt > 0$ small enough so that the following
	holds. For all pairs of locations $x,y\in\locations_s$
	such that $x$ is an ancestor of $y$
	and $\time(x,y) \leq \time^*$: the probability
	that, on a gene tree, no recipient or donor location 
	lies on $p_s(x,y)$ is at least
	$1 - \delta^*$.
	\end{claim}
	\begin{proof}
	For all $z$ on $p_s(x,y)$, let 
	$N_z = |\ball{z}{}|$ be the number of contemporaneous
	locations to $z$. Let $z_1,\ldots,z_k$ be the locations
	on $p_s(x,y)$ where $N_z$, as a function of $z$, has jumps
	and let $N^0,\ldots,N^k$ be the values of $N_z$ on
	the segments so obtained. Let $z_0 = x$ and $z_{k+1} = y$.
	The recipient locations on $p_s(x,y)$ form a
	nonhomogeneous Poisson process with rate bounded by
	$\maxlgt$. On ther other hand, for $i = 0,\ldots, k$,
	the donor locations on
	$(z_i,z_{i+1})$ also form an independent nonhomogeneous
	Poisson process, which can be thought of as
	the superposition (over the contemporaneous branches)
	of thinned nonhomogeneous Poisson processes
	(where the thinning accounts for the choice of donor branch). The total rate of that
	process is bounded above by 
	$$
	(N^i-1)\times \frac{1}{N^i} \times \maxlgt \leq \maxlgt,
	$$ 
	where the first term on the LHS counts the number of
	contemporaneous branches to $(z_i,z_{i+1})$
	not equal to $(z_i,z_{i+1})$ (recall that, under our model, recipient and donor locations cannot
	coincide), the second term is the probability
	of picking $(z_i,z_{i+1})$ as donor, and
	the last term bounds the rate of transfer.
	The donor
	processes on $(z_i,z_{i+1})$, $i = 0,\ldots,k$, are
	independent by the independent increments property
	of Poisson processes. Hence, overall, the transfer locations
	(both recipient and donor) on $p_s(x,y)$ form a Poisson
	process with rate bounded above by $2 \maxlgt$.
	The probability of observing no transfer location
	on $p_s(x,y)$ is therefore at least $e^{-2\maxlgt\time^*}$.
	Taking $\maxlgt$ small enough gives the result.
	\end{proof}

	Fix  $\time_0 = {d_0 \over \sub}$. 
	We proceed to show the claims of the lemma.
	We first show that each gene tree distance
	satisfies the desired bound with high enough
	probability.
	\begin{itemize}
		\item \emph{Short distances:} By definition, $\weight_s(u,v) \le d_0$ implies $\time(u,v) \le \time_0$. Hence, if $x=\mrca(u,v)$, $\time(u,x)=\time(v,x) \le \time_0/2$. During the generation of a gene tree $T_g$ from $T_s$ according to the process of Definition~\ref{def:randomlgt}, by Claim~\ref{claim:no-transfer} with probability at least $0.99$ (given our operating assumption that $\maxlgt$ is small enough), no recipient locations between $x$ and $u$ or between $x$ and $v$ are chosen. Then it follows from Definition~\ref{def:randomlgt} that the resulting gene tree $T_g$ satisfies $\weight_g(u,v) = \weight_s(u,v)$, hence the distorted gene tree satisfies $\weight_g'(u,v)=\weight_s(u,v) \pm \epsilon$.
		
		\item \emph{Long distances:} Suppose $\weight_s(u,v) > d_0$. Viewing $T_s$ as a continuous object, let $x \neq y \in T_s$ be the unique points (guaranteed to exist and be distinct) such that $\weight_s(x,u)=\weight_s(y,v)=d_0/2$. It follows that, $\time(x,u), \time(y,v) \le \time_0/2$. Let now $\bar{x}$ (resp. $\bar{y}$) be the closest ancestor of $x$ (resp. $y$) that belongs to $V_s$. Then, $\time(\bar{x},u), \time(\bar{y},v) \le \time_0/2 +\maxtime$. During the generation of a gene tree $T_g$ from $T_s$ according to the process of Definition~\ref{def:randomlgt}, by Claim~\ref{claim:no-transfer} with probability at least $0.99$ (given our operating assumption that $\maxlgt$ is small enough), no recipient locations between $\bar{x}$ and $u$ or between $\bar{y}$ and $v$ are chosen. Then, by Definition~\ref{def:randomlgt}, the resulting gene tree $T_g$ will contain nodes $\bar{x}, \bar{y}$, the path between $u, v$ in $T_g$ will go through these nodes, and $\weight_g(\bar{x},u) = \weight_s(\bar{x},u)$ and $\weight_g(\bar{y},v) = \weight_s(\bar{y},v)$. Hence, $\weight_g(u,v) \ge d_0$ and the distorted tree satisfies $\weight_g'(u,v)>d_0-\epsilon$.
	\end{itemize}
	Given that a gene tree generated according to the process of Definition~\ref{def:randomlgt} satisfies the claims of the lemma with probability at least $0.99$, the lemma follows from the choice of $N$, and standard concentration bounds~\cite{MotwaniRaghavan:95}.
\end{prevproof}

\begin{prevproof}{Lemma}{lem:trees-from-distorted-metrics}
	We apply single-linkage clustering, as described above the statement of Lemma~\ref{lem:trees-from-distorted-metrics},
	\emph{up to distance $d_0-\epsilon$}. More precisely, we start with each leaf being in a cluster of its own
	with the distance $\hat{d}$ as defined in~\eqref{eq:dhat}. At each iteration, we merge the two closest clusters. When a new cluster is formed, we update $\hat{d}$ by letting the distance between the new cluster $A$ and any other remaining cluster $B$ be defined as 
	\begin{equation}\label{eq:dhat-clusters}
	\hat{d}(A,B) := \min_{a \in A, b\in B} \hat{d}(a,b).
	\end{equation}
	We stop when no pair of clusters is at distance at most $d_0 - \epsilon$.
	
	Let $\mathcal{C}_s$ be the set of all clusters of $T_s$ and 
	let $\mathcal{C}_s[M]$ be those clusters in $\mathcal{C}_s$ whose elements are at distance at most $M$ under $\weight_s$. We claim that
	the algorithm described above reconstructs
	a collection of clusters $\widehat{\mathcal{C}}$
	which satisfies
	\begin{equation}\label{eq:widehatc}
	\widehat{\mathcal{C}}
	=
	\left\{A \in \mathcal{C}_s\,:\, \min_{a_1, a_2 \in A} \hat{d}(a_1,a_2) \leq d_0 - \epsilon\right\},
	\end{equation}
	and, furthermore,
	\begin{equation}\label{eq:widehatc-sandwich}
	\mathcal{C}_s[d_0-2\epsilon] 
	\subseteq \widehat{\mathcal{C}} 
	\subseteq \mathcal{C}_s[d_0].
	\end{equation}
	Note that the sets in $\widehat{\mathcal{C}}$
	are then nested, as those in $\mathcal{C}_s[d_0]$ are nested.
	These conditions together ensure that the output 
	is equivalent to a $(d_0,\epsilon)$-truncation of $T_s$ as in Definition~\ref{def:truncation}.
	The claim follows from Lemma~\ref{lem:one step} 
	and an induction argument on the steps of the algorithm. 
	See for example~\cite[Theorem 1 (Supplementary Materials)]{Roch:10} for such an argument. We omit the details.
\end{prevproof}

\begin{prevproof}{Lemma}{lem:exists-diluted-subtree}
	Recall that, under the HGT process, a subtree
	moves away from its location in the species phylogeny
	if it is the recipient location of an HGT event. By our assumptions and Claim~\ref{claim:no-transfer}, the probability that this event occurs on any given edge of the species phylogeny is bounded by a constant, which can be made arbitrarily small. Hence, we can think of the
	subtree of $u \downarrow T_s$ which is conserved
	under the HGT process as a \emph{percolation process},
	where an edge is open (independently from the other edges) if it does not contain a recipient location of the HGT process. All other edges are said to be closed. The open subtree of $u \downarrow T_s$ then corresponds to a subtree which is shared between the species phylogeny and the gene tree. The result then follows directly by adapting Lemmas~6--8 in~\cite{Mossel:01}.
\end{prevproof}

\begin{prevproof}{Proposition}{lem:generalized one step}
	Consider the generation of gene tree $T_g$ from $T_s$. According to Lemma~\ref{lem:exists-diluted-subtree}, with probability at least $1-2 \delta$, for $\delta = \eta/4$, there exist diluted subtrees $\tilde{T}_u$ of $u \downarrow T_s$ and $\tilde{T}_v$ of $v \downarrow T_s$ that do not receive any recipient locations in the process of Definition~\ref{def:randomlgt}. We condition on this event in the remainder. By the definition of the HGT process, this means that $T_g$ contains trees $T_u$ and $T_v$ that are leafsomorphic to $\tilde{T}_u$ and $\tilde{T}_v$ respectively. Moreover, these trees are rooted at nodes $u$ and $v$ of $T_g$ (which we identify with the corresponding nodes of $T_s$). 
	In particular observe that, for any leaf $\ell_1$ of $\tilde{T}_u$ (and $T_u$), we have that 
	\begin{equation}\label{eq:general step:toroot}
	\weight_s(\ell_1,u) = \weight_g(\ell_1,u),
	\end{equation}
	and similarly for $\tilde{T}_v$.
	In addition, 
	by an analysis analogous to the proof of Lemma~\ref{lem:one step}, it follows that, with probability at least $1-2\delta$, independently w.r.t. the event considered above:
	\begin{align}
	&\text{If $\weight_s(u,v)\le d_0$, then $\weight_g(u,v) = \weight_s(u,v)$;} \label{eq:general step:close}\\
	&\text{If $\weight_s(u,v)>d_0$, then $\weight_g(u,v) > d_0$.}~~~~~~~~~~~~ \label{eq:general step:far}
	\end{align}
	Finally, given that $T_g'$ is leafsomorphic to $T_g$, it follows that $T_g'$ will contain real subtrees $T_u',T_v'$ that are leafsomorphic to $\tilde{T}_u$ and $\tilde{T}_v$, respectively. 
	
	\smallskip 
	Let $T_u''$ and $T_v''$ be arbitrary real subtrees of $T_g'$ that are leafsomorphic to some diluted subtree $\hat{T}_u$ of $u \downarrow T_s$ and some diluted subtree $\hat{T}_v$ of $v \downarrow T_s$ respectively.
	We claim that $T_u'$ and $T_u''$
	have the same root, and similarly for $T_v'$ and $T_v''$. To show this, we first notice the following:
	\begin{claim} \label{claim:little claim}
		$\hat{T}_u$ and $\tilde{T}_u$ share two disjoint paths from their common root $u$ to a pair of shared leaves. The same is true for $\hat{T}_v$ and $\tilde{T}_v$.
	\end{claim}
	\begin{prevproof}{Claim}{claim:little claim}
		Follows immediately by the Pigeonhole principle and the diluted tree degree requirements.
	\end{prevproof}
	Consider the pair of disjoint paths $p_1$, $p_2$ shared by $\tilde{T}_u$ and $\hat{T}_u$ by Claim~\ref{claim:little claim}. Suppose $p_1$ connects the root $u$ to some leaf $\ell_1$ and $p_2$ connects the root $u$ to some leaf $\ell_2$. Since $T_u'$ is leafsomorphic to $\tilde{T}_u$, it must also contain disjoint paths from its root to leaves $\ell_1$ and $\ell_2$. The same is true for $T_u''$, as it is leafsomorphic to $\hat{T}_u$. Since $T_u'$ and $T_u''$ are real subtrees of the same tree $T_g'$ they must have the same roots. Similarly, $T_v'$ and $T_v''$ have the same roots. 
	So we have established the first two claims of the proposition. 
	
	It remains to prove the distance claims.
	Suppose that $\weight_s(u,v) \leq d_0$. For any leaf $\ell_1$ of
	$T_u$ and $\ell_1'$ of $T_v$ we have by~\eqref{eq:general step:toroot} and~\eqref{eq:general step:close} that
	\begin{equation}\label{eq:diluted-distance-1}
	\weight_g(\ell_1,\ell_1') = \weight_s(\ell_1,\ell_1') = \weight_s(A_u,A_v).
	\end{equation}
	Suppose that $u'$ and $v'$ are the roots of $T_u'$ and $T_v'$ in $T_g'$ as defined above. 
	Let again $T_u''$ and $T_v''$ be arbitrary real subtrees of $T_g'$ that are leafsomorphic to some diluted subtree $\hat{T}_u$ of $u \downarrow T_s$ and some diluted subtree $\hat{T}_v$ of $v \downarrow T_s$ respectively. Let $\ell_2$ be any leaf of $T_u''$ and
	$\ell_2'$ be any leaf of $T_v''$. Because the roots
	of $T_u'$ and $T_u''$ and of $T_v'$ and $T_v''$
	coincide, we have further that
	\begin{equation}\label{eq:diluted-distance-2}
	\weight_g(\ell_1,\ell_1') = \weight_g(\ell_2,\ell_2'),
	\end{equation}
	where note that we used the ``undistorted'' metric $\weight_g$,
	which is ultrametric.
	On the other hand, given that $T_g'$ is an $\epsilon$-distortion of $T_g$ it follows that for all $x,y \in \{\ell_1,\ell_2,\ell_1',\ell_2'\}$, 
	\begin{equation}\label{eq:diluted-distance-3}
	\weight_g'(x,y)=\weight_g(x,y)\pm \epsilon.
	\end{equation}
	Combining Equations~\eqref{eq:diluted-distance-1},~\eqref{eq:diluted-distance-2} and~\eqref{eq:diluted-distance-3},
	we obtain
	$$
	\weight_g'(\ell_2,\ell_2') = \weight_s(A_u,A_v) \pm \epsilon,
	$$
	as desired.
	A similar argument holds if $\weight_s(u,v) > d_0$.
	We leave out the details.
\end{prevproof}

\begin{prevproof}{Proposition}{lem:identifying diluted tree}
	The diluted subtree can be computed using dynamic programming. Our algorithm proceeds from the leaves of the tree $T'$ towards the root~$u'$. Letting $T=(V,E)$ and $T'=(V',E')$, for each node $v' \in V'$, we identify whether a real subtree of $v' \downarrow T'$ is leafsomorphic to a diluted subtree of a tree $v \downarrow T$ for some node $v$ of $T$. If this is the case, we store the identity of $v$ in some set-valued function $f:V' \rightarrow 2^V$, maintaining $v\in f(v')$, if such a $v$ exists. Before explaining how to compute the function $f$, let us make an easy observation:
	\begin{lemma} \label{lem:identification}
		Suppose $T_1'$ and $T_2'$ are two real subtrees of $T'$ rooted at nodes $w_1',w_2'$ that are leafsomorphic to diluted subtrees $T_1, T_2$ of $w \downarrow T$. Then $w_1'=w_2'$.
	\end{lemma}
	\begin{prevproof}{Lemma}{lem:identification}
		By the Pigeonhole principle and the degree requirements of diluted subtrees, it follows that $T_1$ and $T_2$ share two disjoint paths from their common root $w$ to a pair of shared leaves $\ell_1,\ell_2$. Since $T_1'$ is leafsomorphic to $T_1$ it must also have two disjoint paths from its root $w_1'$ to leaves $\ell_1,\ell_2$. The same is true for $T_2'$. Hence, $w_1'=w_2'$.
	\end{prevproof}
	Let us also introduce a definition.
	\begin{definition}[$3$-Ball]
		Given a node $u$ of $T$, its {\em $3$-Ball} denoted $B_T(u,3)$ is the subgraph of $T$ containing $u$'s children (if any), grandchildren (if any), and great-grandchildren (if any). A subgraph of $B_T(u,3)$ is called an {\em almost $3$-Ball of $u$} if it is the same as $B_T(u,3)$, except that it might be missing a single node at depth $3$ from the root, if any.
	\end{definition}
	
	\smallskip We are now ready to compute $f$. Given Lemma~\ref{lem:identification}, for all $v \in V$, there is at most one node $v'\in V'$ such that $v \in f(v')$. The initialization of $f$ at the leaves of $T'$ is clear:
	$$v \in f(v') \Leftrightarrow \left(\begin{minipage}[h]{6.3cm}  $v'$ is a leaf of $T'$, $v$ is a leaf of $T$, and $v, v'$ have the same labels \end{minipage}\right).$$
	For each non-leaf node $v'$ of $T'$, working our way up the tree, we initialize $f(v') =\emptyset$. Then, for all $v \in V$, we set $f(v'):=f(v') \cup \{v\}$ iff the following computation succeeds. 
	\begin{enumerate}
		\item Let ${\cal G}$ be the great-grandchildren of $v$. We check the subtree of $T'$ rooted at $v'$ to identify for each great-grandchild $w \in {\cal G}$ of $v$ its inverse $w'=f^{-1}(w)$, if any, inside the subtree. Recall that such $w'$ is unique, if it exists. If more than one great-grandchild of $v$ fail to have inverses in the subtree of $T'$ rooted at $v'$, we output {\tt failure.} Otherwise let ${\cal I}$ be the set of inverses of great-grandchildren of $v$. \label{algorithm:find I}
		\item For all leaves ${\cal L}$ in $B_T(v,3)$ that are also leaves in $T$ and are at depth $\leq 2$ from $v$, we check to see if they are also leaves in the subtree of $T'$ rooted at $v'$. If any of them fails to be a leaf in the subtree of $T'$ rooted at $v'$, we output {\tt failure}.  \label{algorithm:find L}
		\item For all subsets ${\cal I}' \subseteq {\cal I}$ of size $|{\cal I}'| = |{\cal G}|-1$, we find the minimal subtree $T''$ of $T'$ that includes nodes in ${\cal I}' \cup {\cal L} \cup \{v'\}$. If $T''$ is leafsomorphic (only preserving labels in $\cal L$) to an almost $3$-Ball of $v$ such that, whenever a node $w' \in {\cal I}'$ is mapped to a  node $w$ of the almost $3$-Ball, $w \in f(w')$, we output {\tt success}. If all tried sets ${\cal I}'$ fail (or none of the right size exists), then we output {\tt failure.} If we succeed for some ${\cal I}'$, we also store the corresponding sets ${\cal I'}$, ${\cal L}$, tree $T''$ and leafsomorphism, indexing them by $(v',v)$. (We only need to store these for  one successful ${\cal I}'$, if any.) \label{algorithm: try all possible I'}
	\end{enumerate}
	It is clear from its description that $f$ can be computed in polynomial time in the size of $T'$ and $T$, in a bottom-up fashion.

	When the computation of $f$ is over, we identify the node $u' \in V'$, if any, such that $u \in f(u')$. If no such $u'$ is found, we output that there is no real subtree of $T'$ that is leafsomorphic to a diluted subtree of $T$. If such a $u'$ is found, then we construct a real subtree of $u' \downarrow T'$ that is leafsomorphic to a diluted subtree of $T$, by picking nodes iteratively as follows:
	\begin{itemize}
		\item We pick $u'$ and associate it with $u$, if $u$ is not a leaf.
		\item For each picked node $v'$ of $T'$, we check to see if we have associated a node $v$ of $T$ with $v'$. If not, we do nothing for $v'$. If yes and $v$ is not a leaf, then:
		\begin{itemize}
			\item we pick all the nodes in the stored tree $T''$ indexed by $(v',v)$;
			\item for all nodes in the set ${\cal I}'$ indexed by $(v',v)$ we associate them with their corresponding nodes in $T$ according to the leafsomorphism indexed by $(v',v)$.
		\end{itemize}
	\end{itemize}
	Clearly the above procedure takes time linear in all stored information.
	
	Let us now justify the correctness of the computation of $f$, as well as the returned subtree of $T'$, if any. The correctness of the computation of $f$ can be shown inductively from the leaves. Clearly, the values computed for the leaves are correct. Inductively, suppose all values at the subtree rooted at $v'$ have been computed correctly. Let us argue that the value computed for $v'$ is also correct. 
	
	\begin{itemize}
		\item {\bf No false-negatives.} Suppose there exists a diluted subtree $T_v$ of $v \downarrow T$ that is leafsomorphic to a real subtree $T_{v'}'$ of $v' \downarrow T'$. We will argue that when processing nodes $v',v$ our algorithm will add $v$ to $f(v')$. 
		
		Consider the set ${\cal G}$ of great-grandchildren of $v$. The diluted tree definition implies that there exists a subset ${\cal G}' \subseteq {\cal G}$ of great-grandchildren of $v$ of size $|{\cal G}'|\ge |{\cal G}|-1$, such that all great-grandchildren in ${\cal G}'$ are included in $T_v$. Moreover, by the definition of a diluted tree, the subtree $T_w$ of $T_v$ rooted at some node $w \in {\cal G}'$ is a diluted subtree of $w \downarrow T$. Restricting the leafsomorphism between $T_v$ and $T_{v'}'$ to the subtree $T_w$ of $T_v$, we obtain a subtree $T_{w'}'$ of $T_{v'}'$ rooted at some descendant $w'$ of $v'$ that is leafsomorphic to $T_w$. Hence, assuming by induction that the value of $f$ at $w'$ has been computed correctly, $w \in f(w')$. So Step~\ref{algorithm:find I} of our algorithm will not declare~{\tt failure}, and correctly compute set ${\cal I}$. 
		
		Next, consider the set $L_v$ of all leaves in $T$ at (topological) distance at most $2$ from $v$. It is clear that Step~\ref{algorithm:find L} of our algorithm will set ${\cal L}=L_v$. Moreover, since $T_{v'}'$ is leafsomorphic to $T_v$ it must be that all leaves in ${\cal L}$ are descendants of $v'$ in $T'$. So Step~\ref{algorithm:find L} will also not declare {\tt failure}. 
		
		Finally, for each great-grandchild $w \in {\cal G}'$ that is not a leaf in $T$, let us pick two leaves $\ell_1^w, \ell_2^w$ in $T_w$ such that the paths from $\ell_1^w,\ell_2^w$ to $w$ are disjoint. 
		Such pair of leaves is guaranteed to exist by the diluted tree requirements. Since $T_w$ and $T_{w'}'$ are leafsomorphic, $\ell_1^w, \ell_2^w$ also belong to $T_{w'}'$, where $w' = f^{-1}(w)$, and the paths from these leaves to $w'$ are also disjoint. If $w \in {\cal G}'$ is a leaf in $T$, set $\ell_1^w=\ell_2^w = w$. Now consider the set of leaves ${\cal L}'={\cal L} \cup \left(\cup_{w \in {\cal G}'}\{\ell_1^w,\ell_2^w\}\right)$. These leaves are a subset of the leaves of $T_v$ and $T_{v'}'$. Let ${T}_v|{\cal L}'$ be the restriction of $T_v$ to leaves ${\cal L}'$, i.e. the minimal subtree of $T_v$ that contains the nodes in ${\cal L}' \cup \{v\}$. Since $T_v$ is leafsomorphic to $T_{v'}'$, it must be that ${T}_v|{\cal L}'$ is leafsomorphic to ${T}_{v'}'|{\cal L}'$. By our choice of leaves $\ell_1^w, \ell_2^w$, this means that the tree $T''$ constructed in Step~\ref{algorithm: try all possible I'} will be deemed leafsomorphic to an almost $3$-Ball of $v$. So our algorithm will output {\tt Success}.
		
		\item {\bf No false-positives.} Conversely, we show that, if our algorithm adds $v$ to $f(v')$, then it must be that a real subtree of $v' \downarrow T'$  is isomorphic to a diluted subtree of $v \downarrow T$. This follows almost immediately from the description of our algorithm. For $v$ to be included in $f(v')$, it must be that when our algorithm processes $v', v$, it finds out that all but at most one $w \in {\cal G}$ (the set of great-grandchildren of $v$) have an inverse $f^{-1}(w)$ that is a descendant of $v'$ in $T'$, and moreover, all children and grandchildren of $v$ that are leaves in $T$ are also descendants of $v'$ in $T'$. Let  ${\cal I}$ be the set of inverses computed in Step~\ref{algorithm:find I}, and let ${\cal L}$ be the set of leaves computed in Step~\ref{algorithm:find L}. In Step~\ref{algorithm: try all possible I'} our algorithm finds a subset ${\cal I}' \subseteq {\cal I}$ of size $|{\cal G}|-1$ such that the minimal subtree $T''$ of $T'$ that contains the nodes in ${\cal I}' \cup {\cal L} \cup \{v'\}$ is leafsomorphic (only preserving labels in $\cal L$) to an almost $3$-Ball $AB(v,3)$ of $v$. Do the following operation on $T''$  and $AB(v,3)$: For each leaf $w'$ of $T''$ that belongs to ${\cal I}'$, root at $w'$ a real subtree of $w' \downarrow T'$ that is leafsomorphic to a diluted subtree $T_w$ of $w \downarrow T$, where $w$ is the leaf of $AB(v,3)$ such that $w \in f(w')$. Also, root $T_w$ at leaf $w$ of $AB(v,3)$. Call ${\rm grown}(T'')$ and ${\rm grown}(AB(v,3))$ the trees resulting from the above operations. It is clear that ${\rm grown}(T'')$ and ${\rm grown}(AB(v,3))$ are leafsomorphic, ${\rm grown}(AB(v,3))$ is a diluted subtree of $v \downarrow T$, and ${\rm grown}(T'')$ is a real subtree of $v' \downarrow T'$. So we did well to include $v$ to $f(v')$.
	\end{itemize}

	Given that the computation of $f$ is correct, it is clear from the above analysis that determining a real subtree of $T'$ that is isomorphic to a diluted subtree of $T$ is also done correctly.
\end{prevproof}

\section{Algorithmic result: $\epsilon$-contractions}
\label{sec:proofs-contractions}

In this section we provide the proof of our main algorithmic result, Theorem~\ref{thm:main1b}. (In particular, unlike the previous section, we do {\it not} assume here that the substitution rate is constant.)
Throughout this section, our {\bf Operating Assumptions} are the following:
We are given $\epsilon$-contractions $T_{g_1}',\ldots,T_{g_N}'$ of gene trees $T_{g_1},\ldots,T_{g_N}$, generated independently according to the random HGT model of Definition~\ref{def:randomlgt} from a species phylogeny $T_s = (V_s,E_s;\root,\time)$ with rates of horizontal transfer $\lgt(e)$ and rates of substitution $\sub_g(e)$ satisfying the bounded rates model of Definition~\ref{def:brm}. We assume that $0 \leq \epsilon < \mintime\minsub$. Additionally $N \ge C \log n$ for a large enough constant $C$, and $\maxlgt$ a small enough constant, as required by all the lemmas established in this section. In particular, we will skip stating these assumptions in the statements of all lemmas.

We let $d_s(u,v)$, $d_g(u,v)$ and $d_g'(u,v)$ denote the graph distances
between $u$ and $v$ on $T_s$, $T_g$ and $T_g'$ respectively.
(Under $d_s(u,v)$, we ignore the root of $T_s$.)
Recall that a {\em cherry} is a pair of leaves at graph distance $2$,
that is, a pair of leaves that are ``siblings.'' Similarly to Section~\ref{sec:proofs}, the proof contains several steps:
\begin{enumerate}
	\item {\bf Reconstructing cherries:}
	The first key
	idea is to show that, for each pair of 
	leaves at ``short distance'' in the species
	phylogeny, the median {\em graph} distance across 
	the genes is {\em equal} to the actual graph distance
	in the species phylogeny with high probability (Lemma~\ref{lem:one step-contracted}). We then use the median 
	to reconstruct
	the cherries of the species phylogeny (Lemma~\ref{lem:trees-from-distorted-metrics-contracted}).
	
	\item {\bf Going deeper into the tree:} We then
	bootstrap the previous argument to reach
	deeper into the species phylogeny.
	We adapt
	the diluted approach of Section~\ref{sec:proofs}
	to identify corresponding vertices
	in the gene trees and in the reconstructed parts
	of the species phylogeny. We then
	show how to use such diluted subtrees to estimate
	the graph distance between close-by pairs of vertices 
	deep inside the reconstructed phylogeny (Proposition~~\ref{lem:generalized one step-contracted}).
	
	\item {\bf Computing diluted trees and recursing:} Following the dynamic programming approach of Section~\ref{sec:proofs}
	to compute diluted subtrees,
	the final details of the proof
	are described in Section~\ref{sec:proof-finish-contracted}
	where the main induction step is implemented.
	
\end{enumerate}

The proofs of the lemmas below can be found
in Section~\ref{sec:proofs-lemmas-distorted-contracted}.

\subsection{Reconstructing cherries}

We first show how to reconstruct ``short distances'' from the contracted gene trees. 
We use the fact that, at small enough
rates of HGT, the path between two close-by
leaves is unlikely to be the site of an HGT event. 
We also show that median distance estimates of 
``long distances'' are guaranteed to exceed a 
threshold. Compared to Lemma~\ref{lem:one step}
in the distorted case,
there is a new complication. Because we work with
graph distance we must ensure that, at short distances,
not only is there no transfer on the path between
the leaves of interest, but also that the ``subtrees
hanging from that path'' continue to have a representative 
among the leaves.
\begin{lemma}[Median distances are accurate estimates
	of short distances] \label{lem:one step-contracted}
	For any constant integer $d_0 >0$, under our operating assumptions, for all $u,v \in L$, the following are true with probability at least $1-{1\over {\rm poly}(n)}$:
	\begin{enumerate}
		\item \label{lem:one step:close-contracted} \emph{Short distances.} 
		If $d_s(u,v)\le d_0$, then 
		${\rm median}_{i =1,\ldots,N}\{d_{g_i}'(u,v) \} = d_s(u,v);$
		\item \label{lem:one step:far-contracted} \emph{Long distances.} If $d_s(u,v)>d_0$, then 
		${\rm median}_{i =1,\ldots,N}\{d_{g_i}'(u,v) \} > d_0.$ 
	\end{enumerate}
\end{lemma}

Our reconstruction algorithm first reconstructs
all cherries using Lemma~\ref{lem:one step-contracted}.
Then it proceeds by reconstructing ``cherries of cherries,''
and so forth. 
\begin{figure}[b!]
	\centering
	\includegraphics[width=4in]{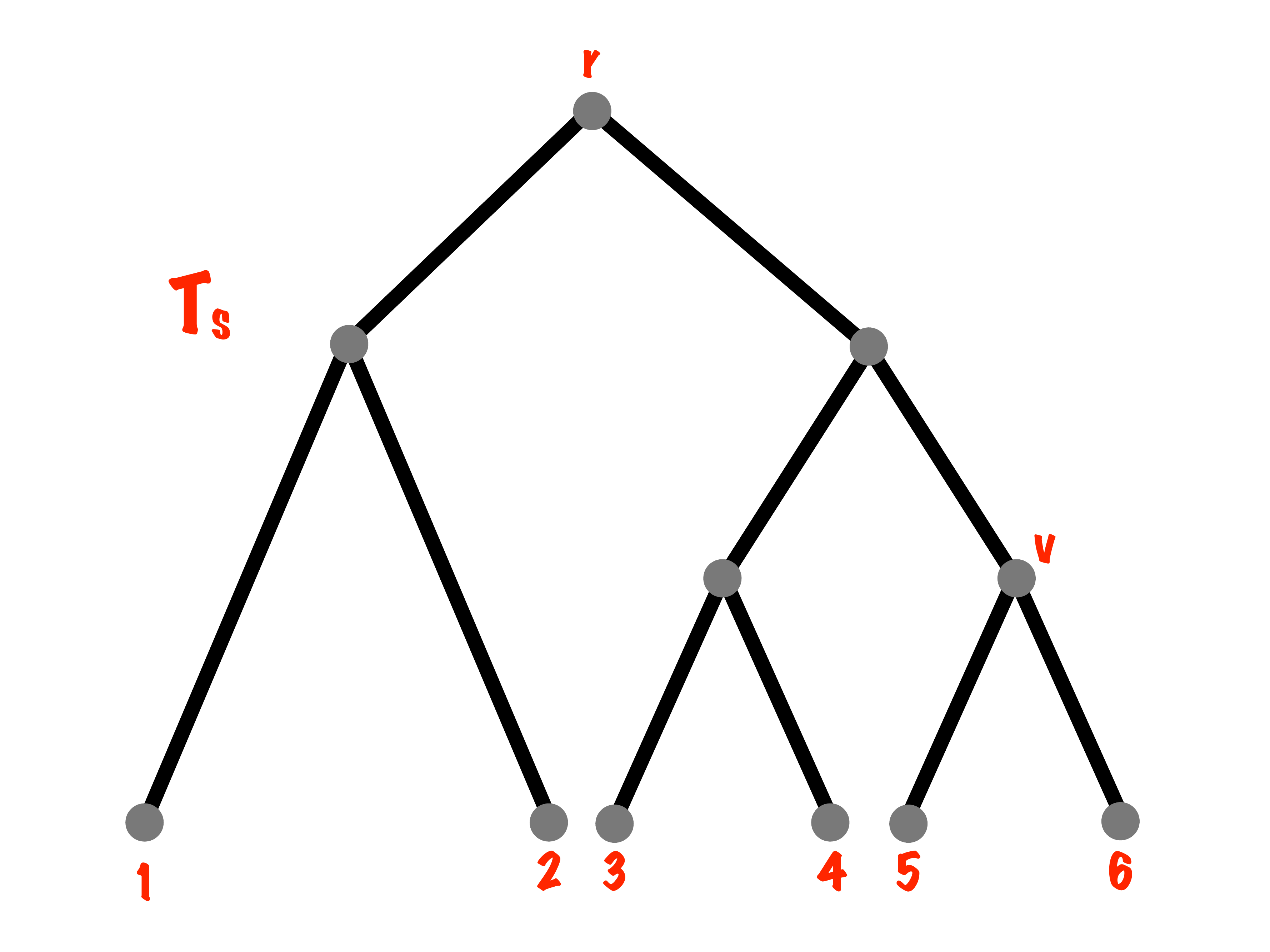}
	\caption{
		A species tree.
	}\label{fig:species-tree}
\end{figure}
\begin{figure}[t!]
	\centering
	\includegraphics[width=6.5in]{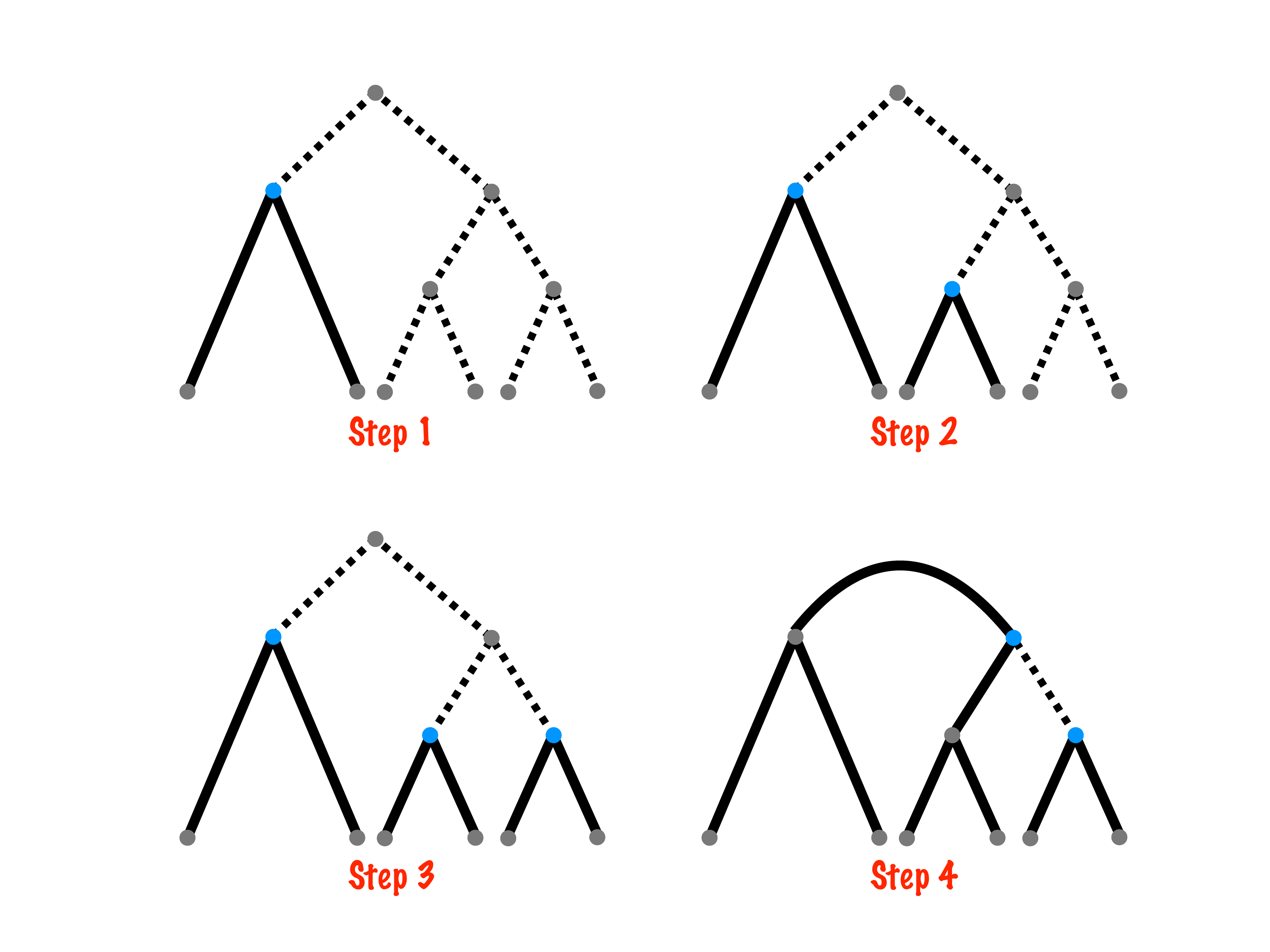}
	\caption{
		Steps of the reconstruction algorithm for the species tree $T_s$ in Figure~\ref{fig:species-tree}. Roots of
		the prunings are in blue. Note that in the
		last step, re-rooting $T_s$ at $v$ shows that 
		the left subtree in the pruning satisfies the
		fullness condition.
	}\label{fig:algo}
\end{figure}
However, as the illustration of our reconstruction
method in Figures~\ref{fig:species-tree} and~\ref{fig:algo}
shows, we cannot guarantee that the reconstruction 
is consistent with the rooting in the species tree.
Instead, we maintain what we call a {\em pruning},
as defined next.
\begin{definition}[Pruning of a phylogeny]
	\label{def:pruning-contracted}
	Given a phylogeny $T_s = (V_s,E_s;\root,\time)$ with leaf-set $L=[n]$ and some integer $D > 0$, a {\em $D$-pruning of $T_s$} is a collection $\forest$ of rooted subtrees $T_i=(V_i,E_i;\root_i)$, $i=1,\ldots,\ell$, of $T_s$ satisfying the following properties:
	\begin{itemize}
		\item \emph{Disjoint forest.} The trees $T_i$, $i=1,\ldots,\ell$, are disjoint,
		that is, do not share edges as subtrees of $T_s$.
		
		\item \emph{Size.} The number of edges in $\forest$,
		as a subforest of $T_s$, is $D$. 
		
		\item \emph{Fullness.} For all $i=1,\ldots,\ell$, the
		tree $T_i$, as a subtree of $T_s$, is {\em full}. That is,
		there is a neighbor $\root_i'$ of $\root_i$ in $T_s$
		such that, if
		$T_s$ were (re-)rooted at $\root_i'$, then $T_i$ would correspond exactly to
		the subtree of $T_s$ rooted at $\root_i$.
		
	\end{itemize}
\end{definition}

We reconstruct an initial pruning using the median estimator:
\begin{equation}\label{eq:dhat-contracted}
\forall u,v \in L,\qquad \hat{d}(u,v) := {\rm median}_{i =1,\ldots,N}\{d_{g_i}'(u,v) \}.
\end{equation}
Formally:
\begin{lemma}[Building an initial pruning] \label{lem:trees-from-distorted-metrics-contracted}
	Let $D$ be twice the number of cherries of $T_s$.
	Assume that $\hat{d}$, as defined in~\eqref{eq:dhat-contracted}, satisfies the statement of
	Lem\-ma~\ref{lem:one step-contracted}. Then a $D$-pruning of $T_s$ can be computed in polynomial-time. 
\end{lemma}

\subsection{Reaching Deeper into the Past} \label{sec:deeper-contracted}

We now show how to bootstrap the median estimator in~\eqref{eq:dhat-contracted}. 
We rely on the concepts of diluted and conserved subtrees
defined in Section~\ref{sec:deeper}.
As in the previous subsection, the main new
hurdle is that the use of the graph distance
requires controlling the HGTs in the
``subtrees hanging from the path'' between
two internal vertices of interest. We use
ideas from percolation for that purpose.
\begin{proposition}[Induction step: Diluted subtrees and distance estimates] \label{lem:generalized one step-contracted} 
	Let $\forest = \{T_1,\ldots,T_\ell\}$ be a pruning of $T_s$.
	Con\-sider 
	constants $d_0, \eta>0$ and a pair of distinct roots $\root_i, \root_j$ of $\forest$ (with respective trees $T_i$ and $T_j$). Under our operating assumptions, a contracted gene tree $T_g'$ satisfies the following with probability at least $1-\eta$:
	\begin{itemize}
		\item \emph{Diluted subtree at $\root_i$.} $T_g'$ contains a real subtree $T_i'$ rooted at some node $\root_i'$ that is leafsomorphic to a diluted subtree of $T_i$; moreover, any such subtree $T_i''$ has the same root $\root_i'$;
		\item \emph{Diluted subtree at $\root_j$.} $T_g'$ contains a real subtree $T_j'$ rooted at some node $\root_j'$ that is leafsomorphic to a diluted subtree of $T_j$; moreover, any such subtree $T_j''$ has the same root $\root_j'$.
	\end{itemize}
	Moreover, for any such subtrees $T_i''$ and $T_j''$:
	\begin{itemize}
		\item \emph{Short distances.} 
		If $d_s(\root_i,\root_j)\le d_0$ then $d_g'(\root_i',\root_j') 
		= d_s(\root_i,\root_j)$; \label{lem:general step:close-contracted}
		\item \emph{Long distances.} 
		If $d_s(\root_i,\root_j)>d_0$ 
		then $d_g'(\root_i',\root_j') > d_0$. \label{lem:general step:far-contracted}
	\end{itemize}
\end{proposition} 

\subsection{Computing Diluted Subtrees}

It remains to show how to compute diluted subtrees.
\begin{proposition}[Induction step: Computing diluted subtrees] \label{lem:identifying diluted tree-contracted}
	Given some node $\root_i'$ of 
	$T_g'$ and a tree $T_i$ from a pruning $\forest$ of $T_s$,
	we can identify in polynomial-time a real subtree of $T_g'$
	rooted at $r_i'$ that is leafsomorphic to a diluted subtree of $T_i$, if such a subtree exists in $T_g'$.
\end{proposition}

\subsection{Phylogenies from Contracted Gene Trees}
\label{sec:proof-finish-contracted}

Using Propositions~\ref{lem:generalized one step-contracted} and~\ref{lem:identifying diluted tree-contracted}, we are now ready
to prove Theorem~\ref{thm:main1b}.

\smallskip

\begin{prevproof}{Theorem}{thm:main1b}
	For every pair $\root_i, \root_j \in V_s$ and full, disjoint subtrees $T_i$, $T_j$ of $T_s$ rooted respectively at $\root_i$ and $\root_j$, it follows from Proposi\-tion~\ref{lem:generalized one step-contracted} and standard concentration inequalities~\cite{MotwaniRaghavan:95} that, with probability at least $1-{1 \over {\rm poly}(n)}$:
	\begin{equation}\label{eq:final2-contracted}
	d_s(\root_i,\root_j)\le d_0 \implies
	{\rm median}_{\ell \in N_{\root_i,\root_j}}\{d_{g_\ell}'(\root_i,\root_j) \} = d_s(\root_i,\root_j);
	\end{equation} 
	\begin{equation}\label{eq:final1-contracted}
	d_s(\root_i,\root_j)>d_0 \implies
	{\rm median}_{\ell \in N_{\root_i,\root_j}}\{d_{g_\ell}'(\root_i,\root_j) \} > d_0;
	\end{equation} 	
	where $N_{\root_i,\root_j}$ is the subset of contracted gene trees that contain a diluted subtree of $T_i$ and of $T_j$. 
	Since there are $O(n^2)$ pairs of $u,v \in V_s$, by a union bound, Equations~\eqref{eq:final2-contracted} and~\eqref{eq:final1-contracted} simultanenously hold for all pairs of $u,v \in V_s$, with probability at least $1-{1 \over {\rm poly}(n)}$. We condition on this event.

	We now describe our high-level reconstruction
	algorithm. We proceed similarly to
	Lemma~\ref{lem:trees-from-distorted-metrics-contracted}.
	But whenever a new subtree is formed, we identify diluted subtrees of this subtree in the contracted gene trees and use these diluted subtrees to estimate inter-root distances
	using the median as above. More precisely:
	\begin{enumerate}
		\item 
		For all $u \in [n]$, let $T_{\{u\}}$ be the tree
		composed of only $u$, with root $\rho_{\{u\}} = u$.
		Let $\forest = \{T_{\{u\}}\,:\, u \in [n]\}$. Set $\hat{d}$ as in~\eqref{eq:dhat-contracted}. 
		
		\item Until the number of edges in $\forest$ is $2n - 4$:
		\begin{enumerate}
			\item Let $T_A, T_B$ be two subtrees in $\forest$ with $\hat{d}$-distance between their roots equal to $2$.
			\item Let $T_{A\cup B}$ be the tree obtained by attaching the trees $T_A$ and $T_B$ at the roots by
			a cherry.
			Let $\rho_{A\cup B}$ be the 
			root of $T_{A\cup B}$, that is, the middle vertex of the new cherry.			
			\item Update $\mathcal{F}$ by removing
			$T_A,T_B$ and adding $T_{A\cup B}$.
			\item For each gene $i \in N$, compute 
			a real subtree $\tilde{T}_{A\cup B}^i$ of $T_{g_i}'$ that is leafsomorphic to a diluted subtree of $T_{A\cup B}$, if such a subtree exists,
			as detailed in the proof of Proposition~\ref{lem:identifying diluted tree-contracted}.
			Let $\tilde{\rho}_{A\cup B}^i$ be the root of $\tilde{T}_{A\cup B}^i$.
			\item Update: for each
			$T_F \in \mathcal{F}$ with $F \neq A\cup B$,
			set
			$$
			\hat{d}(\rho_F,\rho_{A\cup B})
			:= {\rm median}_{i \in N_{\rho_F,\rho_{A\cup B}}}\{d_{g_i}'(\tilde{\rho}_{F}^i,\tilde{\rho}_{A\cup B}^i)\}, 
			$$
			where $N_{\rho_F,\rho_{A\cup B}}$ is the subset of contracted gene trees that contain a diluted subtree of $T_F$ and of $T_{A\cup B}$.
		\end{enumerate}
		\item Add an edge connecting the roots of the two trees in $\forest$.
	\end{enumerate}
	
	At initialization, $\forest$ is a $0$-pruning. 
	By~\eqref{eq:final2-contracted} and~\eqref{eq:final1-contracted},
	at each step we identify a cherry of $T_s$
	where the trees in $\forest$ have been pruned.
	Such a cherry always exists because, by fullness
	and disjointess of the pruning, the above operation produces
	a binary tree. Moreover, after adding a cherry
	as described in the algorithm, we preserve
	fullness and disjointness and the number of edges
	in $\forest$ grows by $2$. The process therefore
	terminates with the topology of $T_s$ in a 
	polynomial number of steps.
	That concludes the proof.
\end{prevproof}

\subsection{Proofs}
\label{sec:proofs-lemmas-distorted-contracted}


\begin{prevproof}{Lemma}{lem:one step-contracted}
	Fix  $\time_0 = d_0 \maxtime $. As the proof is similar
	to that of Lemma~\ref{lem:one step}, we only summarize the argument.
	\begin{itemize}
		\item \emph{Short distances:} Suppose that $d_s(u,v) \le d_0$. Then $\time(u,v) \le d_0 \maxtime  = \time_0$. Arguing as in the proof of Lemma~\ref{lem:one step},
		by Claim~\ref{claim:no-transfer} for $\maxlgt$ small enough there is no transfer on $p_s(u,v)$ with probability at least $0.99$. Note, however, that this is
		not enough to guarantee the result. In particular, if an entire subtree hanging from the path from $u$ and $v$ is tranferred away, then the ancestor of this subtree on $p_s(u,v)$ is not present as a vertex in the gene tree, which decreases the graph distance between $u$ and $v$. We argue that such problematic transfers do not occur with high probability. Let $x=\mrca(u,v)$, let $w_0, w_1,\ldots,w_\ell$
		be the vertices on $p_s(u,v)$ with $w_0 = u$ and $w_\ell = v$, and note that $\ell \leq d_0$. For $i = 1,\ldots,k-1$ such that $w_i \neq x$, let also $y_i$ be a leaf descendant of $w_i$ such that the path $p_s(w_i,y_i)$ does not
		intersect $p_s(u,v)$. For $i$ such that $w_i = x$,
		let $y_i$ be a leaf descendant of the parent of $x$
		such that $p_s(y_i,w_i)$ does not
		intersect $p_s(u,v)$.
		Then,
		letting $y_0 = y_\ell = x$, we have
		$$
		\max\{\time(w_i,y_i)\,:\,i=0,\ldots,\ell\} 
		\le \time_0/2+ 2 \maxtime,
		$$
		where the worst case is achieved for $w_i = x$ when all branch lengths are $\maxtime$.
		Arguing as in the proof of Lemma~\ref{lem:one step},
		by Claim~\ref{claim:no-transfer} for $\maxlgt$ small enough there is no transfer on $p_s(w_i,y_i)$ for 
		all $i=0,\ldots,\ell$ with probability at least $0.99$.
		Then it follows that the resulting gene tree $T_g$ satisfies $d_g'(u,v) = d_g(u,v) = d_s(u,v)$.

		\item \emph{Long distances:} The argument in the
		case of long distances is similar to the proof
		of Lemma~\ref{lem:one step}, as modified in the 
		short distances case above. The details are left out.
	\end{itemize}
	The lemma then follows from standard concentration bounds~\cite{MotwaniRaghavan:95}.
\end{prevproof}

\begin{prevproof}{Lemma}{lem:trees-from-distorted-metrics-contracted}
	We take $d_0 = 2$. Lemma~\ref{lem:one step-contracted} immediately implies that all cherries of $T_s$ can be identified with high probability. The collection $\forest$
	of these cherries forms a $D$-pruning.
\end{prevproof}

\begin{prevproof}{Proposition}{lem:generalized one step-contracted}
	The first part of the claim follows from the argument in
	the proof of Proposition~\ref{lem:generalized one step}.
	For the second part, we argue similarly to Lemma~\ref{lem:one step-contracted}. We only detail the short distances case.
	The other case is similar.
	
	Fix  $\time_0 = \maxtime d_0$.
	Suppose that $d_s(\root_i,\root_j) \le d_0$. 
	Then $\time(\root_i,\root_j) \le \time_0$. 
	Let $w_0, w_1,\ldots,w_\ell$
	be the vertices on $p_s(\root_i,\root_j)$ 
	with $w_0 = \root_i$ and $w_\ell = \root_j$ 
	and note that $\ell \leq d_0$.
	For $i = 1,\ldots,k-1$, 
	let $Y_i$ be the subtree of $T_s$ hanging from $p_s(\root_i,\root_j)$ at $w_i$ (where we think
	of $T_s$ as being unrooted). By Claim~\ref{claim:no-transfer},
	for $\maxlgt$ small enough, the
	probability of a transfer on any given edge of $Y_i$
	can be made arbitrarily small. If we imagine running
	a percolation process on $Y_i$ where an edge is
	open if there is no transfer involving that edge,
	then with constant probability arbitrarily close to $1$
	(for $\maxlgt$ small enough) there exists an open
	path from $w_i$ to some leaf $y_i$ of $Y_i$. We choose
	$\maxlgt$ such that this holds with probability
	at least $1-\eta/2$ for 
	all $i=0,\ldots,\ell$. Then it follows that the resulting gene tree $T_g$ satisfies $d_g'(\root_i',\root_j') = d_g(\root_i',\root_j') = d_s(\root_i,\root_j)$.
\end{prevproof}

\begin{prevproof}{Proposition}{lem:identifying diluted tree-contracted}
	For each neighbor $r_i''$ of $r_i'$ in $T_g'$,
	root $T_g'$ at $r_i''$ and apply the procedure
	described in Proposition~\ref{lem:identifying diluted tree}
	to $T_i$ and $r_i' \downarrow T_g'$.
\end{prevproof}

\section{Impossibility result}
\label{sec:proofs-2}

We now prove Theorem~\ref{thm:main2}.
Similarly to~\cite{RochSnir:12},
our improved impossibility result
uses a coupling argument. Specifically,
we run the HGT processes jointly on
two different phylogenies simultaneously 
and show that
they output the same gene tree with high 
probability. See, e.g.,~\cite{Lindvall:92}
for more on coupling. Our construction
also uses percolation on trees techniques.
See e.g.~\cite{Peres:99} for more on percolation.

\smallskip

\begin{prevproof}{Theorem}{thm:main2}
	Fix $\rho_\lambda = 0$, $\rho_\tau = 1$,
	and $\bar{\tau} = 1$. 
	Let $T$ be a complete binary tree with
	$n = 2^H$ leaves labeled $\{1,\ldots,n\}$
	and with fixed edge lengths $\bar{\tau}$. 
	Let $\overline{T}$ be the same tree as $T$
	with the same leaf labels,
	except for the following change: in
	the canonical planar representation of 
	$T$, swap the first and third subtrees, $\mathcal{T}_1$ and $\mathcal{T}_3$, of $T$
	on level $\frac{2}{3} \log_2 n$ (which for simplicity we assume 
	is integer-valued) from the
	root. Denote by $L_1$ and $L_3$ the respective leaf sets of $\mathcal{T}_1$ and $\mathcal{T}_3$ in $T$.
	Similarly, we let $\overline{\mathcal{T}}_1$ and $\overline{\mathcal{T}}_3$ be the first and
	third subtrees on level $\frac{2}{3} \log_2 n$
	of $\overline{T}$ with respective
	leaf sets $L_3$ and $L_1$.
	Observe that
	$|L_1| = |L_3| = n^{1/3}$.
	Fix $\lambda(e) = 0$ for all edges not in
	$\mathcal{T}_1$ and $\mathcal{T}_3$ and 
	let $\lambda(e) = \bar{\lambda}$
	for all edges in $\mathcal{T}_1$ and $\mathcal{T}_3$.
	Do the same on $\overline{T}$.
	
	We couple the HGT processes in $T$ and
	$\overline{T}$ as follows. We first run the process
	on $T$. The HGT events in $\overline{T}$ 
	are picked as follows: any transfer in $T$
	can be described by the leaf set $L_R$ of the recipient
	location, the leaf set $L_D$ of the donor location
	and the distance from the root; for any such transfer
	in $T$, we perform the exact same transfer in $\overline{T}$, i.e., using the same distance from the root 
	and the same sets $L_R$, $L_D$. Observe that this
	is always possible because at any fixed time for the root, $T$ and $\overline{T}$ share 
	the same subtrees, although some of them are
	arranged differently. By symmetry, this process
	is then
	a coupling of the two HGT processes. 
	We show below that, with probability at least 
	$1-1/2N$, the produced gene trees
	are identical. That implies the theorem.
	
	We make a series of claims. 
	\begin{itemize}
		\item {\bf No in-moves.}
		Note that only the subtrees of $\mathcal{T}_1$/
		$\mathcal{T}_3$ and $\overline{\mathcal{T}}_1$/ $\overline{\mathcal{T}}_3$ can be transferred,
		as the HGT rate is $0$ everywhere else.
		Let $L_{-\{1,3\}} = [n] - L_1\cup L_3$.
		We define two types of transfer. In an
		\emph{out-move}, $L_D \subseteq L_{-\{1,3\}}$. 
		In an \emph{in-move}, $L_D \subseteq L_1\cup L_3$.
		By the definition of the process,
		for any given transfer, the probability
		that it is an in-move is $2/n^{2/3}$.
		For a constant $0 < \bar{\lambda} < +\infty$, the total HGT weight $\lgttotalextinct$ of $\mathcal{T}_1$ and $\mathcal{T}_3$
		is $\Theta(n^{1/3})$. 
		Here we used that a binary tree with
		$n^{1/3}$ leaves has $O(n^{1/3})$
		edges. 
		Because the
		number of transfers is Poisson
		with mean $\lgttotalextinct$, for any
		$\alpha > 0$,
		the probability that more than
		$n^{1/3 + \alpha}$ transfers occur
		overall is at most $O(n^{-\alpha})$
		by Markov's inequality~\cite{MotwaniRaghavan:95}. 
		The probability that any transfer is
		an in-move is then at most
		$O(n^{-\alpha} + n^{-1/3+\alpha})
		= O(n^{-1/6})$
		by the law of total probability,
		where we chose $\alpha = 1/6$.
		Let $\mathcal{E}_1$ be the event
		that there is no in-move. 
		
		\item {\bf Existence of a cut.}
		We say that there is \emph{transfer
			cut} in $\mathcal{T}_1$
		if, for each leaf $\ell_1$ in
		$\mathcal{T}_1$, there is at least one
		transfer on the path between $\ell_1$
		and the root of $\mathcal{T}_1$; and
		similarly for $\mathcal{T}_3$. 
		Let $p_{\bar{\lambda}}$ 
		be the probability that a transfer
		occurs on an edge of $\mathcal{T}_1$.
		Note that, by choosing $\bar{\lambda}$
		large enough (but constant), we can
		make $p_{\bar{\lambda}}$ to be a constant as close to $1$
		as we desire. We associate to the
		HGT process a percolation process to show that,
		for $\bar{\lambda}$ large enough,
		a transfer cut exists in both
		$\mathcal{T}_1$ and $\mathcal{T}_3$
		with high probability. Consider $\mathcal{T}_1$. We say that an edge
		of $\mathcal{T}_1$ is \emph{closed} if 
		it contains the recipient location
		of at least one transfer. Otherwise
		it is \emph{open}. Let $L_1'$ be the subset
		of $L_1$ connected to the root of
		$\mathcal{T}_1$ by an open path.
		Because each edge is open
		independently with probability $1-p_{\bar{\lambda}}$, 
		the expected size of $L_1'$ is 
		$(1-p_{\bar{\lambda}})^{H'} n^{1/3}$ 
		where $H' = \frac{1}{3} \log_2 n$.
		By Markov's inequality again, the probability
		that $L_1'$ is non-empty, i.e., that
		$|L_1'| \geq 1$, is at most
		$(1-p_{\bar{\lambda}})^{H'} n^{1/3}
		= O(n^{-1/6})$ by choosing
		$\bar{\lambda}$ 
		to be a large enough constant.
		The same holds for $\mathcal{T}_3$.
		Let $\mathcal{E}_2$ be the event
		that there is a transfer cut
		in both $\mathcal{T}_1$ and
		$\mathcal{T}_3$. 
		
		\item {\bf Same output.} Condition
		on the events $\mathcal{E}_1$ and
		$\mathcal{E}_2$, which are guaranteed
		to occur simultaneously with
		probability at least $1 - O(n^{-1/6})$.
		The existence of transfer cuts 
		and the absence of in-moves
		imply that all leaves of $\mathcal{T}_1$
		and $\mathcal{T}_3$ (and similarly for 
		$\overline{\mathcal{T}}_1$ and
		$\overline{\mathcal{T}}_3$) 
		have been transferred into
		the shared part of $T$ and $\overline{T}$
		(and have possibly subsequently moved 
		\emph{within}
		the shared part).
		Because under our coupling the donor locations of the
		transfers are chosen to be the same in
		$T$ and $\overline{T}$, the output
		gene trees are then identical.
		
	\end{itemize}	
	That concludes the proof.
\end{prevproof}

\clearpage

\bibliographystyle{alpha}
\bibliography{own,thesis,RECOMB12}



\end{document}